\documentclass[a4paper,12pt]{article}
\usepackage[utf8]{inputenc}
\usepackage{mathtools}
\usepackage{enumerate}
\usepackage{url}

\title{Numerical methods for changing type systems}
\author{Sebastian Franz\footnote{
           Institut f\"ur Numerische Mathematik, Technische Universit\"at Dresden,
           01062 Dresden, Germany.
           \mbox{e-mail}: sebastian.franz@tu-dresden.de}\quad
        Sascha Trostorff\footnote{
           Institut f\"ur Analysis, Technische Universit\"at Dresden,
           01062 Dresden, Germany.
           \mbox{e-mail}: sascha.trostorff@tu-dresden.de}\quad
        Marcus Waurick\footnote{
           Department of Mathematical Sciences, University of Bath,
           Bath, UK.
           \mbox{e-mail}: m.waurick@bath.ac.uk}
        }
\date{\today}

\usepackage{fancyhdr} 
\usepackage{pgfpages}
\usepackage{amsmath}
\usepackage{amssymb}
\usepackage{amsthm}

\makeatletter
\let\my@saved@original@eqref\eqref 
\renewcommand*{\eqref}[1]{
  \begingroup
    \let\normalfont\relax
    \my@saved@original@eqref{#1}
  \endgroup
}
\makeatother

\usepackage{cite}

\usepackage{color}

\newcommand{\e}{\mathrm{e}}

\newcommand{\grad}{\nabla}

\DeclareMathOperator{\Div}{{div}}

\newcommand{\jump}[1]{[\hspace*{-2pt}[#1]\hspace*{-2pt}]}

\newcommand{\norm}[2]{\|{#1}\|_{#2}}
\newcommand{\snorm}[2]{|{#1}|_{#2}}

\newcommand{\tnorm}[1]{\left|\!\!\;\left|\!\!\;\left| {#1}
                       \right|\!\!\;\right|\!\!\;\right|}

\newcommand{\scp}[1]{\langle #1 \rangle_H}

\newcommand{\N}{\mathbb{N}}

\newcommand{\R}{\mathbb{R}}

\newcommand{\U}{\mathcal{U}}
\newcommand{\PS}{\mathcal{P}}
\newcommand{\QS}{\mathcal{Q}}

\newcounter{tmp}
\newcommand{\makeballnumber}[1]{\setcounter{tmp}{\theenumi}%
\setcounter{enumi}{#1}%
\leavevmode \csname beamer@@tmpl@enumerate item\endcsname%
\setcounter{enumi}{\thetmp}}
\newcommand{\makeball}{\leavevmode \csname beamer@@tmpl@itemize item\endcsname}

\definecolor{seb}{rgb}{0.9,0,0}

\newcommand{\vn}{\boldsymbol{n}}

\newcommand{\vx}{\boldsymbol{x}}

\newcommand{\dt}{\,\mathrm{d}t}

\makeatletter
\renewcommand*\env@matrix[1][r]{\hskip -\arraycolsep
  \let\@ifnextchar\new@ifnextchar
  \array{*\c@MaxMatrixCols #1}}
\makeatother

\numberwithin{equation}{section}

\usepackage{tikz}
\usepackage{pgfplots}
\usetikzlibrary{calc}
\usetikzlibrary{external}
\tikzsetfigurename{figure_}
\tikzset{external/system call={pdflatex \tikzexternalcheckshellescape -interaction=batchmode -jobname "\image" "\texsource";
convert -density 600 -transparent white "\image.pdf" "\image.png"}}

\renewcommand{\boldsymbol}[1]{\mathbf{#1}}
\renewcommand{\phi}{\varphi}
\newcommand*{\dive}{\operatorname{div}}
\DeclareMathAlphabet{\mathcal}{OMS}{cmsy}{m}{n}
\newcommand{\m}{\mathrm{m}}
\theoremstyle{plain}
\newtheorem{thm}{Theorem}[section]
\newtheorem{lem}[thm]{Lemma}

\newtheorem{prop}[thm]{Proposition}
\newtheorem{cor}[thm]{Corollary}

\theoremstyle{definition}
\newtheorem*{deff}{Definition}
\newtheorem{rem}[thm]{Remark}
\newtheorem{exam}[thm]{Example}

\usepackage{bbm}
\newcommand{\lin}{\mathrm{lin}\,}
\newcommand{\id}[1]{\chi_{#1}}
\newcommand{\scpr}[1]{\langle #1 \rangle_\rho}
\newcommand{\scprm}[1]{\langle #1 \rangle_{\rho,m}}
\newcommand{\dd}{\,\mathrm{d}}
\newcommand{\tmi}{t_{m,i}}
\newcommand{\Qm}[1]{Q_m\left[#1\right]}
\newcommand{\Qmr}[1]{Q_m\left[#1\right]_\rho}

\begin{document}
 \pagestyle{fancy}
  \maketitle
  \begin{abstract}
     In this note we develop a numerical method for partial differential equations with changing type. Our method is based on a unified solution theory found by Rainer Picard for several linear equations from mathematical physics. Parallel to the solution theory already developed, we frame our numerical method in a discontinuous Galerkin approach in space-time with certain exponentially weighted spaces.
  \end{abstract}



 \section*{Acknowledgements}
  M.W.~carried out this work with financial support of the EPSRC frant EP/L018802/2: ``Mathematical foundations of metamaterials: homogenisation, dissipation and operator theory''. This is gratefully acknowledged.
 \section{Introduction}
   Following the rationale presented in \cite{Picard}, most of the classical linear partial differential equations arising in mathematical physics share a common form, namely the form of an evolutionary problem. That is, we consider equations of the form
  \begin{equation}\label{eq:evo}
   (\partial_tM_0+M_1+A)U=F,
  \end{equation}
where $F$ is a given source term, $\partial_t$ stands for the derivative with respect to time,  $M_0,M_1$ are bounded linear operators on some Hilbert space $H$ and $A$ is an unbounded skew-selfadjoint operator in $H$. We are seeking for a unique solution $U$ of the above equation. We remark here that we do not impose initial conditions, since we consider the whole real line as time horizon, and hence, we implicitly assume a vanishing initial value at ``$-\infty$''. To illustrate the setting, we begin with presenting some examples.
\begin{exam}\label{ex:illus}
 Let $\Omega\subseteq \mathbb{R}^n$ an open non-empty set, where $n\in \N$, but, typically $n\in \{1,2,3\}$. We define the following two differential operators  
 \[
  \grad_0: H_0^1(\Omega)\subseteq L^2(\Omega)\to L^2(\Omega)^n, 
 \]
assigning each function $u\in H_0^1(\Omega)$ its gradient, that is, the column-vector of its partial derivatives in each coordinate direction. Moreover, we set  
\[
 \dive\coloneqq -(\grad_0)^\ast :D(\dive)\subseteq L^2(\Omega)^n\to L^2(\Omega),
\]
which is nothing but the operator assigning each $L^2$ vector-field its distributional divergence with maximal domain, that is,
\[
 D(\dive)=\{ v\in L^2(\Omega)^n\, :\, \sum_{i=1}^n \partial_i v_i \in L^2 (\Omega)\}.
\]
Since both the operators $\grad_0$ and $\dive$ are closed and skew-adjoints of one another, we infer that the operator 
\[
 A\coloneqq \begin{pmatrix}
             0 & \Div \\
             \grad_0 & 0
            \end{pmatrix}:D(\grad_0)\times D(\dive)\subseteq L^2 (\Omega)\times L^2(\Omega)^n\to L^2 (\Omega)\times L^2 (\Omega)^n
\]
is skew-selfadjoint on the Hilbert space $H=L^2(\Omega)\times L^2 (\Omega)^n$. Choosing $M_0=1$ and $M_1=0$ in \eqref{eq:evo}, the corresponding evolutionary problem reads as 
\[
 \left(\partial_t +\begin{pmatrix}
             0 & \dive \\
             \grad_0 & 0
            \end{pmatrix}\right) \begin{pmatrix}
                                  u\\
                                  v
                                  \end{pmatrix} =\begin{pmatrix}
                                                  f\\
                                                  g
                                                  \end{pmatrix}.
\]
If $g=0$, this is nothing but the wave equation. Indeed, the second line then gives $\partial_t v=-\grad_0u$, and hence, differentiating the first line with respect to time, we obtain 
\[
\partial_{t}^2u-\dive\grad_0 u=\partial_{tt}u+\dive \partial_t v=\partial_t f.
\]
Note that $\dive\grad_0=\Delta_{\textnormal{D}}$ is the classical Dirichlet--Laplace operator on $L^2(\Omega)$.

Choosing $M_0=\begin{pmatrix}
                   1 & 0 \\
                   0 & 0
                  \end{pmatrix}$ and $M_1=\begin{pmatrix}
                                       0 & 0 \\
                                       0 & 1 
                                       \end{pmatrix}$ in \eqref{eq:evo}, the corresponding problem reads as 
\[
 \left(\partial_t\begin{pmatrix}
                   1 & 0 \\
                   0 & 0
                  \end{pmatrix}+\begin{pmatrix}
                                       0 & 0 \\
                                       0 & 1 
                                       \end{pmatrix}+\begin{pmatrix}
             0 & \dive \\
             \grad_0 & 0
            \end{pmatrix}\right)\begin{pmatrix}
                                  u\\
                                  v
                                  \end{pmatrix} =\begin{pmatrix}
                                                  f\\
                                                  g
                                                  \end{pmatrix}.
\]
Setting again $g=0$, the latter gives the heat equation. Indeed, the second line reads $v=-\grad_0u$ and hence the first line yields
\[
 \partial_t u-\dive\grad_0u=\partial_t+\dive v=f.
\]
Finally, choosing $M_0=0$ and $M_1=1$ in \eqref{eq:evo}, we get
\[
 \left(1+\begin{pmatrix}
             0 & \dive \\
             \grad_0 & 0
            \end{pmatrix}\right)\begin{pmatrix}
                                  u\\
                                  v
                                  \end{pmatrix} =\begin{pmatrix}
                                                  f\\
                                                  g
                                                  \end{pmatrix},
\]
which in the case $g=0$ gives the elliptic equation
\[
 u-\dive \grad_0u=f.
\]
\end{exam}

\begin{rem}
 We note that we can treat the case of homogeneous Neumann boundary conditions in the same way. The only difference is that we define $\grad$ as the distributional gradient on $H^1(\Omega)$ and $\dive_0\coloneqq -(\grad)^\ast$. Replacing now $\grad_0$ by $\grad$ and $\dive$ by $\dive_0$ yields the same hyperbolic, parabolic and elliptic type problem above, but now with homogeneous Neumann boundary conditions. 
\end{rem}

Example \ref{ex:illus} shows that evolutionary problems cover all three classical types of partial differential equations, elliptic, parabolic and hyperbolic. However, also problems of mixed type are covered as the next example shows.
\begin{exam}\label{ex:CTS}
 Recall the setting of Example \ref{ex:illus}. We decompose $\Omega$ into three measurable, disjoint sets $\Omega_{\mathrm{ell}},\Omega_{\mathrm{par}}$ and $\Omega_{\mathrm{hyp}}$ and set $M_0=\begin{pmatrix}
                                                                                                                                               \chi_{\Omega_{\mathrm{hyp}}\cup\Omega_{\mathrm{par}}} & 0\\
                                                                                                                                              0 &\chi_{\Omega_{\mathrm{hyp}}}
                                                                                                                                              \end{pmatrix}$ as well as $M_1=\begin{pmatrix}
                                                                                                                                                                               \chi_{\Omega_{\mathrm{ell}}} & 0 \\
                                                                                                                                                                               0 & \chi_{\Omega_{\mathrm{par}}\cup\Omega_{\mathrm{ell}}}
                                                                                                                                                                               \end{pmatrix}$. The resulting evolutionary problem then is of mixed type. More precisely, on $\Omega_{\mathrm{ell}}$ we get an equation of elliptic type, on $\Omega_{\mathrm{par}}$ the equations becomes parabolic while on $\Omega_{\mathrm{hyp}}$ the problem is hyperbolic.\end{exam}
                                                                                                                                                                               
\begin{rem}
  The interested reader might wonder that there is not imposed any transmission condition on the unknown quantities along the interfaces of $\Omega_{\mathrm{ell}},\Omega_{\mathrm{par}}$ and $\Omega_{\mathrm{hyp}}$. However, this can be implemented automatically by being in the domain of the corresponding operator sum, as can be seen, for instance, in \cite[Remark 3.2]{W_StH16}, see also \cite[An illustrative Example]{PTWW13}. Another example of a mixeed tyoe problem in contral theory can be found in \cite[Remark 6.2]{PTW_IMA16}
\end{rem}
                                                                                                                                                                               
In \cite{Picard}, the well-posedness of problems of the form \eqref{eq:evo} has been addressed. In fact, it was shown that these probolems also cover the classical Maxwell's equations, the equations of linearized elasticity or a general class of coupled phenomena, see, for instance, \cite{MPTW16,MPTW15,PTWZ15}.  All these problems are indeed well-posed (see Section 2 for the precise statement). The purpose of the present article is to provide numerical methods for such problems. In this article, for the applications to follow, we will focus, however, on problems of mixed type of the form sketched in Example \ref{ex:CTS}. Moreover, as the spatial discretisation has to be developed for each problem separately, anyway, in this work, we will put an emphasize on the time-discretisation. Furthermore, we want to stress that the null-space of $M_0$ in \eqref{eq:evo} might be infinite-dimensional. Hence, we seek to develop a numerical scheme, which in particular allows for the treatment of a certain class of (partial) differential-\emph{algebraic} equations. 

  For the numerical treatment of the time derivatives we use a discontinuous Galer\-kin (dG) method,
  see also Section~\ref{s:sd}.
  The first dG-method was published in 1973 on neutron transport \cite{RH73}. 
  Later the methodology was developed further for classical hyperbolic, parabolic 
  and elliptic problems, see also the survey article \cite{Cockburn2000} and the 
  book \cite{Riviere08}. Note that there is a strong connection between dG-methods 
  and Runge-Kutta (collocation) methods, see \cite{AkrivisMakridakisNochetto2011} 
  for parabolic problems.

In Section \ref{s:ep}, for convenience, we will recall some essentials for evolutionary equations. In particular, we recall the solution theory of problems of the type of equation \eqref{eq:evo}. We will introduce a semi-dicretised version, Equation \eqref{eq:discr_quad_form}, of equation \eqref{eq:evo} at the beginning of Section \ref{s:sd}. We will also provide a solution theory for this semi-discretised variant with general underlying (spatial) Hilbert space $H$ (Proposition \ref{prop:solvability}). The remainder of Section \ref{s:sd} is devoted to estimate difference of the exact solution of \eqref{eq:evo} and the approximate solution of \eqref{eq:discr_quad_form}: In Subsection \ref{su:err1}, we bound the error by solely in terms of the interpolation error, which will eventually be estimated in Subsection \ref{su:ip}. As our prime example, we address the full space-time discretisation of Example \ref{ex:CTS} and derive corresponding error estimates. We verify our theoretical findings in Section \ref{s:ne} by means of a $1+1$- and a $1+2$-dimensional numerical example. This article is attached an appendix (Section \ref{s:app}), where, for the convenience of the reader, we recall some well-known results on the Gau\ss--Radau quadrature rule including the fact that the choice of Gau\ss--Radau points depends continuously on the weighting function. We will need some implications of the fact just mentioned in our a-priori analysis in Subsection \ref{su:err1}.

   \section{The setting of evolutionary problems}\label{s:ep}
   
   In this section we briefly recall the well-posedness result stated in \cite{Picard}. For doing so, we need to specify the functional analytic setting. Throughout, let $H$ be a real Hilbert space.
   
   \begin{deff}
    Let $\rho> 0$ and define the space 
    \[
     H_\rho(\mathbb{R};H)\coloneqq \{ f:\mathbb{R}\to H\,:\, f \mbox{ meas.}, \intop_\mathbb{R} |f(t)|_H^2 \exp(-2\rho t) \dd t<\infty\},
    \]
where we as usual identify functions which are equal almost everywhere. The space $H_\rho(\mathbb{R};H)$ is a Hilbert space endowed with the natural inner product given by 
\[
 \langle f,g\rangle_\rho \coloneqq \intop_{\mathbb{R}} \langle f(t),g(t)\rangle_H \exp(-2\rho t) \dd t \quad(f,g\in H_\rho(\mathbb{R};H)).
\]
Moreover, we define $\partial_{t}$ to be the closure of the operator
\[
 \partial_t :C_c^\infty(\mathbb{R};H)\subseteq H_\rho(\mathbb{R};H)\to H_\rho(\mathbb{R};H): \phi\mapsto \phi',
\]
where by $C_c^\infty(\mathbb{R};H)$ we denote the space of infinitely differentiable $H$-valued functions on $\mathbb{R}$ with compact support. We denote the domain of $\partial^k_t$ by $H_\rho^k(\mathbb{R};H)$ for $k\in \mathbb{N}$.
\end{deff}

Within the setting introduced, we can formulate the well-posedness for evolutionary equations of the form \eqref{eq:evo}.

\begin{thm}[{{\cite[Solution Theory]{Picard}}}]\label{thm:Pic}
Let $M_0,M_1:H\to H$ be bounded linear operators, $M_0$ selfadjoint and $A:D(A)\subseteq H\to H$ skew-selfadjoint. Moreover, assume that there is some $\rho_0>0$ such that
\[
 \exists \gamma>0 \forall \rho\geq \rho_0,x\in H: \langle (\rho M_0 +M_1) x,x\rangle_H\geq \gamma\langle x,x\rangle_H.
\]
Then, for each $\rho\geq \rho_0$ and each $F\in H_\rho(\mathbb{R};H)$ there exists a unique $U\in H_\rho(\mathbb{R};H)$ such that
\begin{equation}\label{eq:evo_1}
 \overline{(\partial_t M_0+M_1+A)}U=F,
\end{equation}
where the closure is taken in $H_\rho(\mathbb{R};H)$. Moreover, the following continuity estimate holds
\[
 |U|_\rho \leq \frac{1}{\gamma} |F|_\rho.
\]
If $F\in H_\rho^k(\mathbb{R};H)$ for $k\in \mathbb{N}$, then so is $U$ and we can omit the closure bar in \eqref{eq:evo_1}.
\end{thm}

\begin{rem}\label{rem:reg} $\,$
\begin{enumerate}[(a)]
             \item Note that the positive definiteness condition in the latter theorem especially implies $\langle M_0 x,x\rangle_H \geq 0$ for each $x\in H$.
             \item We remark that $H_\rho^1(\mathbb{R};H)\hookrightarrow C_\rho(\mathbb{R};H)$ by a variant of the Sobolev embedding theorem \cite[Lemma 3.1.59]{PM} or \cite[Lemma 5.2]{KPSTW14}. Here, 
             \[
              C_\rho(\mathbb{R};H)\coloneqq\{ f:\mathbb{R}\to H\,:\,f \mbox{ cont., } \sup_{t\in \mathbb{R}} |f(t)|\exp(-\rho t)<\infty\}.
             \]
            \item If $F\in H_\rho^1(\mathbb{R};H)$ then $U\in H_\rho^1(\mathbb{R};H)$ and hence 
            \[
             AU=F-\partial_tM_0U-M_1U\in H_\rho(\mathbb{R};H),
            \]
            which yields that $U(t)\in D(A)$ for almost every $t\in \mathbb{R}$. If even $F,U\in H_\rho^2 (\mathbb{R};H)$ the latter gives $AU\in H_\rho^1(\mathbb{R};H)$ and hence, using the Sobolev embedding result (see part (b)), $U\in C_\rho(\mathbb{R};D(A))$. 
            \item The original result in \cite{Picard} treat a general class of time-translation invariant coefficients. We refer to \cite{PTWW13,W15} for non-autonomous variants as well as to \cite{T12,TW14} for non-autonomous and/or non-linear versions of Theorem \ref{thm:Pic}.
            \end{enumerate}
\end{rem}

We note that the equations treated in Example \ref{ex:illus} and Example \ref{ex:CTS} satisfy the conditions of the previous theorem and hence, are well-posed.

\section{Semi-discretisation in time}\label{s:sd}

In this section, we discretise \eqref{eq:evo} with respect to time and do the a-priori analysis. We assume that $A,M_0,M_1$ satisfy the assumptions of Theorem \ref{thm:Pic}. Let $\rho\geq \rho_0$ and fix $T>0$ and consider the time interval $[0,T]$ instead of the whole real line. 
We partition the time-interval $[0,T]$ into subintervals $I_m=(t_{m-1},t_m]$ of length $\tau_m$ for $m\in\{1,2,\ldots,M\}$ with $t_0=0$ and $t_M=T$. Let $q\in \mathbb{N}$. We define the space 
\[
 \mathcal{U}^\tau\coloneqq \{u\in H_\rho(\mathbb{R};H)\,:\,\forall m\in \{1,\ldots,M\}: u|_{I_m}\in \mathcal P_q(I_m;H)\},
\]
where we denote by 
\[
\mathcal P_q(I_m;H)\coloneqq \lin \{ I_m \ni t\mapsto t^k \zeta \in H; k\in \{0,\ldots,q\}, \zeta\in H\}
\]
 the space of $H$-valued polynomials of degree at most $q$ defined on $I_m$. We endow $\mathcal P_q(I_m;H)$ with the scalar product 
 \[
     \langle p,q\rangle_{\rho,m} \coloneqq \intop_{t_{m-1}}^{t_m} \langle p(t),q(t)\rangle_H \exp(-2\rho(t-t_{m-1})) \dd t
    \]
turning the space $\mathcal P_q(I_m;H)$ into a Hilbert space.

The time integrals have to be evaluated numerically. We choose on each time interval $I_m$ a right-sided weighted Gau\ss--Radau quadrature formula.
   To this end, denote by ${\omega}_i^m$ and $\hat{t}^m_i$, $i\in\{0,\dots,q\}$,
   the weights and nodes of the weighted Gau\ss--Radau formula with $q+1$ nodes
   on the reference time interval $\widehat{I}=(-1,1]$, such that
   \[
      \int_{\widehat I} \e^{-\rho \tau_m (t+1)}p(t)\dt = \sum_{i=0}^q {\omega}^m_ip(\hat{t}^m_i)
   \]
   holds for all polynomials $p$ of degree at most $2q$.
   Note that the weights and nodes can always be numerically computed
   as shown for instance in \cite[Chapter 4.6]{PTVF07}, see also the appendix (Section \ref{s:app}) for some basic facts on the Gau\ss--Radau quadrature.
   With the following standard transformation
   \[
      T_m : \widehat{I}\to I_m,\; \hat{t}\mapsto \frac{t_{m-1}+t_m}{2} + \frac{\tau_m}{2} \hat{t},
   \]
   we define by
   \[
      \Qm{v} := \frac{\tau_m}{2} \sum_{i=0}^q {\omega}^m_i v(\tmi)
   \]
   with the transformed Gau\ss--Radau points $\tmi := T_m(\hat{t}^m_i)$, $i=\{0,\dots,q\}$,
   a quadrature formula on $I_m$.
   Note that
   \[
       \Qm{p} = \int_{I_m}p(t)\e^{-2\rho(t-t_{m-1})}\dt
   \]
   for all polynomials of degree at most $2q$.

   Using
   \[
       \Qmr{a,b}:=\Qm{\scp{a,b}}
   \]
   instead of the scalar products $\langle a,b\rangle_{\rho}$ we employ the following discrete \textbf{quadrature formulation}:

   For given $F\in \U^\tau$ and $x_0\in H$, find $U\in\U^\tau$, such that for all $\Phi\in \U^\tau$ and $m\in\{1,2,\dots,M\}$ it holds
   \begin{gather}\label{eq:discr_quad_form}
     \Qmr{(\partial_t M_0+M_1+A)U,\Phi}
    +\scp{M_0 \jump{U}_{m-1}^{x_0},{\Phi}^+_{m-1}}
    =
    \Qmr{ F,\Phi }.
   \end{gather}
Here, we denote by 
\[
 \jump{U}_{m-1}^{x_0}\coloneqq\begin{cases}
                               U(t_{m-1}+)-U(t_{m-1}-),& m\in\{2,\ldots,M\}
                               \\ U(t_0+)-x_0,& m=1,
                              \end{cases} 
\]
and by $\Phi^+_{m-1}\coloneqq \Phi(t_{m-1}+)$.

   \begin{prop}\label{prop:solvability}
      Let $F\in\U^\tau,x_0\in H$. Then there exists a unique solution of \eqref{eq:discr_quad_form}.
   \end{prop}
   
   \begin{proof}
    Let $m\in\{1,\ldots,M\}$ and recall that $\mathcal{P}_q(I_m;H)$ is a Hilbert space with the aforementioned scalar product.
We note that 
    \[
\partial_t:\mathcal{P}_q(I_m;H)\to \mathcal{P}_q(I_m;H):\,p\mapsto p'     
    \]
 and 
 \[
 \delta_{m-1}: \mathcal{P}_q(I_m;H)\to \mathbb{R}:\, p\mapsto p(t_{m-1}+)
\]
     are bounded linear operators. Consequently, the mapping
     \[
      \mathcal{P}_q(I_m;H)\to \mathbb{R} : p \mapsto \langle x,\delta_{m-1} p\rangle_H
     \]
is linear and bounded for each $x\in H$ and thus, by the Riesz representation theorem, there is a unique $\Psi(x)\in \mathcal{P}_q(I_m;H)$ such that 
\[
 \langle \Psi (x),p\rangle_{\rho,m}=\langle x,\delta_{m-1}p\rangle_H.
\]
Moreover, the mapping $\Psi: H \to \mathcal{P}_q(I_m;H)$ is linear and bounded, since 
\[
 |\Psi(x)|^2_{\rho,m}=\langle \Psi(x),\Psi(x)\rangle_{\rho,m}=\langle x,\delta_{m-1}\Psi(x)\rangle_H\leq |x|_H\|\delta_{m-1}\||\Psi(x)|_{\rho,m}\quad(x\in H).
\]
We now prove that for each $g\in \mathcal{P}_q(I_m;H)$ there is a unique $u\in \mathcal{P}_q(I_m;D(A))$ such that 
\[
 (\partial_t M_0 + M_1 +A)u+\Psi M_0 \delta_{m-1} u=g.
\]
For doing so, we first compute using integration by parts
\begin{align*}
 &\langle\partial_tM_0 v,v\rangle_{\rho,m} \\
 &= \frac{1}{2} \langle \partial_t M_0 v,v\rangle_{\rho,m} + \frac{1}{2} \langle v,\partial_tM_0 v\rangle_{\rho,m} \\
 &= \frac{1}{2} \langle \partial_t M_0 v,v\rangle_{\rho,m} + \frac{1}{2} \intop_{t_{m-1}}^{t_m} \langle v(t),M_0 v'(t) \rangle_H \exp(-2\rho(t-t_{m-1})) \dd t\\ 
 &= \frac{1}{2} \langle \partial_t M_0 v,v\rangle_{\rho,m} - \frac{1}{2} \intop_{t_{m-1}}^{t_m} \langle M_0 v'(t), v(t) \rangle_H \exp(-2\rho(t-t_{m-1})) \dd t\\
 &\quad + \rho \intop_{t_{m-1}}^{t_m} \langle M_0v(t), v(t) \rangle_H \exp(-2\rho(t-t_{m-1})) \dd t
  +\frac{1}{2} \langle M_0 v(t_m),v(t_m)\rangle_H \exp(-2\rho \tau_m)\\
  &\quad - \frac{1}{2} \langle M_0v(t_{m-1}),v(t_{m-1})\rangle_H\\
 &\geq  \rho \langle M_0 v,v\rangle_{\rho,m}-\frac{1}{2}\langle \Psi M_0 \delta_{m-1} v,v\rangle_{\rho,m}  
\end{align*}
for each $v\in \mathcal{P}_q(I_m;H)$. Next, from $A^\ast=-A$ it follows $\langle Ax,x\rangle_H=0$ for each $x\in D(A)$. Therefore, for all $u\in \mathcal{P}_{q}(I_m;D(A))$ we get
\begin{align*}
 &\langle (\partial_t M_0 + M_1 +A)u+\Psi M_0 \delta_{m-1} u,u\rangle_{\rho,m}\\&=\langle \partial_tM_0 u,u\rangle_{\rho,m}+\langle M_1 u,u\rangle_{\rho,m}+\langle \Psi M_0 \delta_{m-1}u,u\rangle_{\rho,m}\\
                                                                             &\geq \langle (\rho M_0 +M_1)u,u\rangle_{\rho,m}+\frac{1}{2}\langle \Psi M_0 \delta_{m-1}u,u\rangle_{\rho,m}\\
                                                                             &\geq \gamma\langle u,u\rangle_{\rho,m},
 \end{align*}
 where we have used 
 \[
  \langle \Psi M_0 \delta_{m-1}u,u\rangle_{\rho,m}=\langle M_0 u(t_{m-1}+),u(t_{m-1})\rangle_H \geq 0.
 \]
In particular, both $B\coloneqq(\partial_t M_0 + M_1 )+\Psi M_0 \delta_{m-1}$ and $B+A$ are strictly positive definite. Moreover, since $B$ is bounded, $B^*$ is strictly positive definite, as well. Hence, from 
\[
 (B+A)^\ast=B^\ast-A
\]
we read off that $(B+A)^*$ is strictly positive definite as well. 
Thus, for each $g\in \mathcal{P}_q(I_m;H)$ there is a unique $u\in \mathcal{P}_q(I_m;D(A))=D(A+B)$ such that 
\begin{equation}\label{eq:p31}
  (\partial_t M_0 + M_1 +A)u+\Psi M_0 \delta_{m-1} u=g
\end{equation}
Thus, we are in the position to define a solution for \eqref{eq:discr_quad_form} by induction on $m$. For this, we put $U(t_0-)\coloneqq x_0$. Next, assume we have solved \eqref{eq:discr_quad_form} for $U$ on $I_{m-1}$ for some $m\in\{1,\ldots,M\}$ ($I_0\coloneqq \{t_0\}$ and the equation is void). Then, let $u\in \mathcal{P}_{q}(I_m;D(A))$ be such that \eqref{eq:p31} holds for $g=F|_{I_m}-\Psi M_0 U(t_{m-1}-)$. We put $U|_{I_m}\coloneqq u$.  The thus defined function $U$ solves \eqref{eq:discr_quad_form}: We observe
\begin{multline*}
 \langle (\partial_t M_0+M_1+A)U,\Phi\rangle_{\rho,m}+ \langle \Psi M_0 \delta_{m-1} U,\Phi\rangle_{\rho,m} \\=\langle F-\Psi M_0 U(t_{m-1}-),\Phi\rangle_{\rho,m}=\langle F,\Phi\rangle_{\rho,m}+ \langle \Psi M_0 U(t_{m-1}-),\Phi\rangle_{\rho,m},
\end{multline*}
by definition for all $\Phi\in \U^\tau$ and $m\in\{1,\ldots,M\}$. The latter is the same as saying 
    \begin{multline*}
     \langle (\partial_t M_0+M_1+A)U,\Phi\rangle_{\rho,m}+ \langle M_0 U(t_{m-1}+),\Phi(t_{m-1}+)\rangle_H\\
     =\langle F,\Phi\rangle_{\rho,m}+ \langle M_0 U(t_{m-1}-),\Phi(t_{m-1}+)\rangle_H.
    \end{multline*}
But, since the quadrature is exact for polynomials up to degree $2q$, the latter equation in turn is equivalent to
\[ \Qmr{(\partial_t M_0+M_1+A)U,\Phi}
    +\scp{M_0 \jump{U}_{m-1}^{x_0},{\Phi}^+_{m-1}}
    =
    \Qmr{ F,\Phi },
    \]
which yields existence of $U$. Uniqueness follows from the uniqueness of $u$ satisfying \eqref{eq:p31}.
\end{proof}

\subsection{On some a-priori error estimates in time}\label{su:err1}

After having proved the unique solvability of \eqref{eq:discr_quad_form}, we address the error estimates in the following. In our analysis we will use the discretised norms
   \[
       \tnorm{v}^2_{Q,\rho,m}:=\Qmr{v,v}
       \quad\text{and}\quad
       \tnorm{v}^2_{Q,\rho}:=\sum_{m=1}^M\Qmr{v,v}\e^{-2\rho t_{m-1}}
   \]
   as approximations of $\tnorm{v}^2_{\rho,m}\coloneqq \intop_{I_m} |v(t)|_H^2 \exp(-2\rho(t-t_{m-1}))\dd t$ and $|v|^2_{\rho}$. Note that for $v\in\U^\tau$ the approximation is exact.

   Let us start by defining an interpolation operator into $\U^\tau$ and
   define by $\varphi_{m,i}$ with $i\in\{0,\dots,q\}$ the associated Lagrange basis functions
   to the nodes $\tmi$. Then we obtain for a function $v\in C([0,T],H)$ by
   \begin{equation}\label{eq:P}
         (Pv)(0) = v(0),\quad
         (Pv)\big|_{I_m}(t)
            = \sum_{i=0}^q v(\tmi)\varphi_{m,i}(t), \qquad m\in\{1,\dots,M\},
   \end{equation}
   an interpolation operator in time.
   
   In the analysis to follow, we will consider the problem \eqref{eq:evo_1}. In particular, we emphasize that we assume that 
   \begin{center}
the hypotheses of Theorem \ref{thm:Pic} are in effect.  \end{center}
 Furthermore, we fix a right-hand side \[F\in H_\rho^2(\mathbb{R};H).\]
 Thus, by Theorem \ref{thm:Pic} (and Remark \ref{rem:reg}(c)) there exists a unique solution \begin{equation}\label{e:op}                           
U\in H_\rho^2(\mathbb{R};H)\text{ with }   
    (\partial_t M_0 +M_1 +A)U=F.
   \end{equation}Also, by Remark \ref{rem:reg}(c), we obtain $F\in C_\rho(\mathbb{R};H)$ and $U\in C_\rho(\mathbb{R};D(A))$. Moreover, we set \begin{center}$U^\tau\in \U^\tau$ to satisfy \eqref{eq:discr_quad_form} for the right-hand side $PF\in \U^\tau$ and  $x_0\coloneqq U(0+)$.\end{center}
   We consider the following splitting
   \[
       U^\tau-U=\xi-\eta,\quad\text{where}\quad
       \xi  = U^\tau-P U\in\U^\tau \quad\text{and}\quad
       \eta =  U - PU.
   \]
   Note that for almost every $t\in [0,T]$ we have that
   \[
    (\partial_t M_0 +M_1+A)U(t)=F(t)
   \]
 and thus, 
 \[
  \langle(\partial_t M_0 +M_1+A)U(t), \Phi(t)\rangle_H=\langle F(t), \Phi (t)\rangle_H
 \]
for each $\Phi\in \U^\tau$ and almost every $t\in[0,T]$, which gives
\[
  \Qmr{(\partial_t M_0+M_1+A)U,\Phi}
    +\scp{M_0 \jump{U}_{m-1}^{x_0},{\Phi}^+_{m-1}}
    =
    \Qmr{ F,\Phi }=\Qmr{ PF,\Phi },
\]
where we have used $M_0 \jump{U}_{m-1}^{x_0}=M_0 \jump{U}_{m-1}^{U(0+)}=0$, due to the continuity of $U$ and $\Qmr{ F,\Phi }=\Qmr{ PF,\Phi }$, since the $PF$ is interpolates at the Gau\ss--Radau points used in the quadrature. Hence, $U$ solves the same semi-discretised problem as $U^\tau$.   Thus, we obtain with $\chi\in\U^\tau$ as test function the \textbf{error equation}
   \begin{multline}\label{eq:error2_}
     \Qmr{(\partial_t M_0+M_1+A)\xi,\chi}
    +\scp{M_0 \jump{\xi}_{m-1}^0,\chi^+_{m-1}}\\
    =
     \Qmr{(\partial_t M_0+M_1+A)\eta,\chi}
    +\scp{M_0 \jump{\eta}_{m-1}^0,\chi^+_{m-1}}.
   \end{multline}
   For the special case $\chi=\xi$ (use $A=-A^*$) we obtain
   \begin{multline}\label{eq:error2}
     E_\textnormal{d}^m
     :=\Qmr{(\partial_t M_0+M_1)\xi,\xi}
    +\scp{M_0 \jump{\xi}_{m-1}^0,\xi^+_{m-1}}\\
    =
     \Qmr{(\partial_t M_0+M_1+A)\eta,\xi}
    +\scp{M_0 \jump{\eta}_{m-1}^0,\xi^+_{m-1}}
    =: E_\textnormal{i}^m
   \end{multline}
 for all $m\in\{1,\ldots,M\}$, where the subscripts d and i should remind of discretisation and interpolation, respectively.
   \begin{lem}\label{lem:lhs2} For all $m\in\{1,\ldots,M\}$, we have
       \[
           E_\textnormal{d}^m \geq \frac{1}{2}\left[
                                \scp{M_0\xi^-_m,\xi^-_m}\e^{-2\rho\tau_m}
                               -\scp{M_0\xi^-_{m-1},\xi^-_{m-1}}
                               +\scp{M_0\jump{\xi}_{m-1}^0,\jump{\xi}_{m-1}^0}
                            \right]
                            +\gamma\tnorm{\xi}_{Q,\rho,m}^2,
       \]
       where $\xi^-_m\coloneqq \xi(t_{m}-)$ and $\xi^-_0\coloneqq 0$.
   \end{lem}
   \begin{proof} Let $m\in\{1,\ldots,M\}$. Since $\xi$ is a (piece-wise) polynomial of order $q$ in time, we obtain
      \begin{align*}
          \Qmr{\partial_t M_0\xi,\xi}
            &=\scprm{\partial_t M_0\xi,\xi}\\
            &=\frac{1}{2}\int_{I_m}\e^{-2\rho(t-t_{m-1})}\partial_t\scp{M_0\xi,\xi}\dt\\
            &=\frac{1}{2}\left[
                           \scp{M_0\xi^-_m,\xi^-_m}\e^{-2\rho\tau_m}
                          -\scp{M_0\xi^+_{m-1},\xi^+_{m-1}}
                         \right]
              +\rho\scprm{M_0\xi,\xi}.
      \end{align*}
      Further, we compute
      \[
         \scp{M_0 \jump{\xi}^0_{m-1},\xi^+_{m-1}}
            = \frac{1}{2}\left[
                               \scp{M_0\xi^+_{m-1},\xi^+_{m-1}}
                              -\scp{M_0\xi^-_{m-1},\xi^-_{m-1}}
                              +\scp{M_0\jump{\xi}^0_{m-1},\jump{\xi}^0_{m-1}}
                           \right].
      \]
      Therefore, we have
      \begin{align*}
         E_\textnormal{d}^m&=\Qmr{(\partial_t M_0+M_1)\xi,\xi}
            +\scp{M_0 \jump{\xi}^0_{m-1},\xi^+_{m-1}}\\
            &= \frac{1}{2}\left[
                               \scp{M_0\xi^-_m,\xi^-_m}\e^{-2\rho\tau_m}
                              -\scp{M_0\xi^-_{m-1},\xi^-_{m-1}}
                              +\scp{M_0\jump{\xi}^0_{m-1},\jump{\xi}^0_{m-1}}
                           \right]\\&\quad
                           +\scprm{(\rho M_0+M_1)\xi,\xi}.
      \end{align*}
      Together with
      \[
          \scprm{(\rho M_0+M_1)\xi,\xi}
          \geq\gamma\tnorm{\xi}_{\rho,m}^2
          = \gamma\tnorm{\xi}_{Q,\rho,m}^2
      \]
      the lemma is proved.
   \end{proof}

   In order to analyse $E_\textnormal{i}^m$ we introduce another interpolation operator, that enables us to
   estimate the time derivative of the interpolation error with a higher order.
   This operator utilises $t_{m,-1}\coloneqq t_{m-1}$ in addition to
   $\tmi$, $i\in\{0,\dots,q\}$ as interpolation points. Denoting the associated Lagrange basis functions by $\psi_{m,i}$,
   $i\in\{-1,0,\dots,q\}$, this interpolation operator is given by
   \begin{equation}\label{eq:hatP}
      (\widehat{P} v)\big|_{I_m}(t)
      := \sum_{i=-1}^{q} v(\tmi)\psi_{m,i}(t),\qquad m\in\{1,\dots,M\}.
   \end{equation}
   Note the $\widehat{P}$ maps to functions that are continuous in time (recall that $t_{m,q}=t_m$) while the image of $P$ is allowed to be
   discontinuous at the time mesh points.

   \begin{lem}\label{lem:rhs2_1}
    For $m\in\{1,\ldots,M\}$ and $\psi\in \U^\tau$, we have
    \[
         \Qmr{\partial_t M_0\eta,\psi}
        +\scp{M_0\jump{\eta}^0_{m-1},\psi^+_{m-1}}
        = \Qmr{\partial_t M_0(U-\widehat{P} U),\psi} + R(U,\psi),
    \]
    where
    \[
        |R(U,\psi)|\leq C\alpha\tau_m|M_0\eta_{m-1}^+|_H^2
                 +\beta\tnorm{\psi}_{Q,\rho,m}^2
    \]
    for all $\alpha,\beta>0$ satisfying $\alpha\beta=1/4$ and with $C\ge0$ depending on $T$ (the finite time horizon) and $\rho$ only.
   \end{lem}
   \begin{proof}
    With $U$ being continuous in time, we only have to consider the discrete part.
    Using the exactness of the quadrature rule for polynomials of degree $2q$, we obtain for $m\in\{1,\ldots,M\}$
    \begin{align*}
     &\Qmr{\partial_t M_0 P U,\psi}+ \underbrace{\scp{M_0\jump{P U}^{x_0}_{m-1},\psi^+_{m-1}}}_{=:a}\\
       &=\scprm{\partial_t M_0 P U,\psi}+a\\
       &=-\scprm{M_0 P U,\partial_t \psi}+2\rho\scprm{M_0 P U,\psi}+\underbrace{\scp{\e^{-2\rho(t-t_{m-1})}M_0 P U,\psi}\big|_{t_{m-1}}^{t_m}}_{=:b}+a\\
       &=-\Qmr{M_0 P U,\partial_t \psi}+2\rho\scprm{M_0 P U,\psi}+a+b\\
       &=-\Qmr{M_0 \widehat{P} U,\partial_t \psi}+2\rho\scprm{M_0 P U,\psi}+a+b\\
       &=-\scprm{M_0 \widehat{P} U,\partial_t \psi}+2\rho\scprm{M_0 P U,\psi}+a+b\\
       &=\scprm{\partial_t M_0 \widehat{P} U,\psi}+2\rho(\scprm{M_0 P U,\psi}-\scprm{M_0 \widehat{P} U,\psi})\\&\quad+a+b-\underbrace{\scp{\e^{-2\rho(t-t_{m-1})}M_0 \widehat{P} U,\psi}\big|_{t_{m-1}}^{t_m}}_{=:c}.
    \end{align*}
    Using
    $(P U)^-_{m-1}=(\widehat{P} U)^+_{m-1}$ ($m\ge2$), $(\widehat{P} U)^+_0=U(0+)=x_0$ and $(P U)^-_m=(\widehat{P} U)^-_m$ ($m\ge1$), we have
    \[
     a+b-c=0.
    \]
    Furthermore, it holds
    \begin{align*}
    2\rho(\scprm{M_0 P U,\psi}-\scprm{M_0 \widehat{P} U,\psi})
          &= 2\rho\scprm{M_0(P-\widehat{P}) U,\psi}\\
          &= 2\rho\scprm{M_0((P-\widehat{P}) U)(t_{m-1}^+)\chi,\psi}\\
          &= 2\rho\scprm{M_0(P U-U)(t_{m-1}^+)\chi,\psi}=:R(U,\psi),
    \end{align*}
    where $\chi\in\PS_{q+1}(I_m)$ with $\chi(t_{m-1})=1$ and $\chi(\tmi)=0$, $i\in\{0,\dots,q\}$. By Corollary \ref{c:chi2} for $K=T$ (note that $0< \tau_m=|I_m|\leq T$), we obtain
    \[
        \int_{I_m} |\chi(t)|^2 e^{-2\rho(t-t_{m-1})}dx\leq C\tau_m 
    \]
    for some $C\geq 0$.
    Thus, we get
    \begin{align*}
     |R(U,\psi)|
        &\leq 2\rho
                      \tnorm{M_0(P U-U)(t_{m-1}^+)\chi}_{\rho,m}
                      \tnorm{\psi}_{\rho,m}
                   \\
        &=2\rho\       |M_0(P U-U)(t_{m-1}^+)|_H \tnorm{\chi}_{\rho,m}
                      \tnorm{\psi}_{\rho,m}
               \\
        &\leq C^2(2\rho)^2\alpha\tau_m|M_0(P U-U)(t_{m-1}^+)|^2
              +\beta\tnorm{\psi}_{Q,\rho,m}^2,
    \end{align*}
    where $\alpha\beta=1/4$. Combining above transformations we are done.
   \end{proof}

   \begin{lem}\label{lem:rhs2_2} For all $m\in\{1,\ldots,M\}$, we have for all $\psi\in\U^\tau$
    \[
        \Qmr{M_1\eta,\psi}=0=\Qmr{A\eta,\psi}.
    \]     
   \end{lem}
   \begin{proof}
    These equalities follow from the fact that $\eta(t_{m,i})=PU(t_{m,i})-U(t_{m,i})=0$ for each $i\in\{0,\ldots,q\}$ and $M_1,A$ are purely spatial operators.
   \end{proof}

   Combining the previous lemmas gives a first result.

   \begin{thm}\label{thm:err2}
    There exists a $C\geq0$ depending on $T$, $\rho$ and $\gamma$, only, such that
    \begin{multline*}
        \scp{M_0\xi^-_M,\xi^-_M} +\e^{2\rho T}\tnorm{\xi}_{Q,\rho}^2
        \leq C\e^{2\rho T}\bigg(
           \tnorm{\partial_t M_0(U-\widehat{P} U)}_{Q,\rho}^2 \\
           +T\max_{1\leq m\leq M}\left\{|M_0\eta^+_{m-1}|_H^2\e^{-2\rho t_{m-1}}\right\}
        \bigg)=:g(U).
    \end{multline*}
   \end{thm}
   \begin{proof}
     Combining Lemmas \ref{lem:lhs2} to \ref{lem:rhs2_2} for $\psi=\xi$ we have for some $C\geq 1$ depending on $T$ and $\rho$ only
     \begin{align*}
      |E_\textnormal{d}^m|&\geq\frac{1}{2}\left[
                                \scp{M_0\xi^-_m,\xi^-_m}\e^{-2\rho\tau_m}
                               -\scp{M_0\xi^-_{m-1},\xi^-_{m-1}}
                               +\scp{M_0\jump{\xi}_{m-1}^{0},\jump{\xi}_{m-1}^{0}}
                            \right]
                            +\gamma\tnorm{\xi}_{Q,\rho,m}^2,\\
      |E_\textnormal{i}^m|&\leq C\alpha\bigg(
                \tnorm{\partial_t M_0(U-\widehat{P} U)}_{Q,\rho,m}^2
                +\tau_m|M_0\eta^+_{m-1}|_H^2
            \bigg)+2\beta\tnorm{\xi}_{Q,\rho,m}^2.
     \end{align*}
     Summing with weights $\e^{-2\rho t_{m-1}}$ for $m\in\{1,\dots,M\}$ we obtain
     \begin{align*}
      \sum_{m=1}^M\e^{-2\rho t_{m-1}}|E_\textnormal{d}^m|
      &\geq\frac{1}{2}\left[
                    \scp{M_0\xi^-_M,\xi^-_M}\e^{-2\rho t_M}
                   -\scp{M_0\xi^-_{0},\xi^-_{0}}\right] \\
                   &\quad\quad+\frac{1}{2}\sum_{m=1}^M\e^{-2\rho t_{m-1}}\scp{M_0\jump{\xi}_{m-1}^0,\jump{\xi}_{m-1}^0}
                  +\gamma\tnorm{\xi}_{Q,\rho}^2,\\
      &\geq \frac{1}{2}\scp{M_0\xi^-_M,\xi^-_M}\e^{-2\rho T}
            +\gamma\tnorm{\xi}_{Q,\rho}^2,
            \end{align*}
      by $\xi^-_{0}=0$ and neglecting the positive jump-contributions, and
      \begin{multline*}
      \sum_{m=1}^M\e^{-2\rho t_{m-1}}|E_\textnormal{i}^m|
      \leq C\alpha\bigg(
                  \tnorm{\partial_t M_0(U-\widehat{P} U)}_{Q,\rho}^2
                                  +T\max_{1\leq m\leq M}\{|M_0\eta^+_{m-1}|_H^2\e^{-2\rho t_{m-1}}\}
            \bigg)\\
            +2\beta\tnorm{\xi}_{Q,\rho}^2.
     \end{multline*}
     Thus for $\beta<{\gamma}/{2}$ the result is proved upon the equality $E_\textnormal{i}^m=E_\textnormal{d}^m$.
   \end{proof}
   \begin{rem}\label{r:err2} Let $m\in \{1,\ldots,M\}$. Note that the estimate in Theorem \ref{thm:err2} remains valid, if one respectively replaces $T$ by $t_m$, $\xi_M^-$ by $\xi_m^-$ as well as the $\tnorm{\cdot}_{Q,\rho}$ by $\tnorm{\cdot\chi_{\tilde{I}}}_{Q,\rho}$ with $\chi_{\tilde{I}}$ being the characteristic function of $\tilde{I}= \bigcup_{k=1}^m I_k$.
   \end{rem}
   In the following, we want to improve Theorem \ref{thm:err2}. In order to do so, we will need the following technical lemmas.
   They are adaptations of \cite[Lemma 2.1 and Corollary 2.1]{AM04}. For the upcoming result and the corresponding proof, we recall for polynomials $a,b\in \mathcal{P}_q(0,1;H)$
   \[
      \scpr{a,b}=\int_0^1 \scp{a(t),b(t)}\, \e^{-2\rho t}\dt
   \]
   and the corresponding integration by parts formula
    \begin{equation}
        \scpr{a',b} = -\scpr{a,b' } +2\rho\scpr{a,b }+\e^{-2\rho t}\scp{a,b}\Big|_0^1.\label{eq:intpart}
    \end{equation}
   \begin{lem}\label{lem:akmak}
    Let $t_i$, $w_i$, $i\in\{0,\dots,q\}$ be the points and
    weights of the right-sided Gau\ss-Radau quadrature rule of order $q$ on $(0,1]$ with weighting function $t\mapsto\e^{-2\rho t}$.

    Let $p\in \PS_{q}(0,1;H)$ and $\tilde p$ the Lagrange interpolant w.r.t.~$(t_i)_{i\in\{0,\ldots,q\}}$ of $\phi \colon (0,1]\ni t\mapsto p(t)/t$. Then
    \begin{align*}
         \scpr{p',\tilde p} + \langle p(0),\tilde p(0)\rangle_H
            &\geq \frac{1}{2}\left(|p(1)|_H^2\e^{-2\rho}+\scpr{\tilde p,\tilde p}\right) +\rho\scpr{p,\tilde p}\\
            &\geq \frac{1}{2}\left(|p(1)|_H^2\e^{-2\rho}+\scpr{p,p}\right) +\rho\scpr{p,p}.
    \end{align*}
   \end{lem}
   \begin{proof}
    Define $v\in\PS_{q-1}((0,1);H)$ by $v(t)=(p(t)-p(0))/t$ and $\Lambda\in\PS_{q}[0,1]$ by
    \[
        \Lambda(t_i)=\frac{1}{t_i},\,i\in\{0,\dots,q\}.
    \]
    Then
    \[
        p(t)=p(0)+t v(t)
        \quad\text{and}\quad
        \tilde p(t)=v(t)+p(0)\Lambda(t).
    \]
    Thus,
    \begin{align*}
     \scpr{p',\tilde p}
        &=\scpr{v+\m v',v+p(0)\Lambda}\\
        &=\scpr{v,v}+
          \scpr{v,\m v'}
          +
                   \scpr{v,p(0)\Lambda}+
                   \scpr{v',p(0) \m \Lambda}
               ,
    \end{align*}
    where we denote by $\m$ the multiplication-with-the-argument, that is, $(\m f)(t)\coloneqq tf(t)$. 
    With \eqref{eq:intpart} we obtain for the second term
    \[
        \scpr{v',\m v}=\frac{1}{2}\left(
                                \e^{-2\rho} |v(1)|_H^2+2\rho\scpr{\m v,v}-\scpr{v,v}
                               \right).
    \]
    From $\m v'\Lambda\in\PS_{2q-1}$ and $\m \Lambda'\Lambda\in\PS_{2q}$ together with the exactness of the quadrature rule it follows that
    \begin{align*}
        \scpr{v',p(0) \m \Lambda}
           &=\sum_{i=0}^q w_i\scp{v'(t_i),p(0)}t_i\frac{1}{t_i}\\
            &=\scpr{v',p(0)} \\
            &= 2\rho\scpr{v,p(0)}+\e^{-2\rho}\scp{v(1),p(0)}-\scp{v(0),p(0)},
\end{align*}
and in the same way
 \begin{align}
  \scpr{\Lambda',\m \Lambda}
           &= 2\rho\scpr{\Lambda,1}+\e^{-2\rho}-\Lambda(0).\label{eq:Lambda0}
 \end{align}
    We have thus this far
    \begin{align*}
     \scpr{p',\tilde p}
        &=\frac{1}{2}\left(
                       \e^{-2\rho}|v(1)|_H^2+
                       \scpr{v,v}+
                       2\rho\scpr{\m v,v}
                     \right)+\\
          &\quad +
                   \scpr{v,p(0)\Lambda}+
                   2\rho\scpr{v,p(0)}
                   +\e^{-2\rho}\scp{v(1),p(0)}-\scp{v(0),p(0)}.
               \end{align*}
    Using $v(1)=p(1)-p(0)$ and $\scp{p(0),\tilde p(0)}=\scp{p(0),v(0)}+\Lambda(0)|p(0)|_H^2$ we obtain
    \begin{multline*}
     \scpr{p',\tilde p}+\scp{p(0),\tilde p(0)}
        =\frac{1}{2}\left(
                       \e^{-2\rho}|p(1)|_H^2+
                       \scpr{v,v}+
                       2\rho\scpr{\m v,v}
                     \right)\\
          +
                   \scpr{v,p(0)\Lambda}+
                   2\rho\scpr{v,p(0)}
         +|p(0)|^2_H\left( \Lambda(0)-\frac{\e^{-2\rho}}{2}\right).
    \end{multline*}
    Next, \eqref{eq:Lambda0} yields
    \begin{align*}
     \Lambda(0)
        &= 2\rho\scpr{\Lambda,1}+\e^{-2\rho}-\scpr{\Lambda',\m\Lambda}\\
        &= 2\rho\scpr{\Lambda,1}-\rho\scpr{\Lambda,\m \Lambda}+\frac{1}{2}\left( \e^{-2\rho}+\scpr{\Lambda,\Lambda}\right)
    \end{align*}
    and hence
    \begin{multline*}
     \scpr{p',\tilde p}+\scp{p(0),\tilde p(0)}
        = \frac{1}{2}\left(
                       \e^{-2\rho}|p(1)|_H^2+
                       \scpr{v,v}+
                       2\rho\scpr{\m v,v}
                     \right)\\
          +
                   \scpr{v,p(0)\Lambda}+
                   2\rho\scpr{v,p(0)}
         +|p(0)|^2_H\left( \rho\scpr{\Lambda,2-\m\Lambda}+\frac{1}{2}\scpr{\Lambda,\Lambda}\right)
    \end{multline*}
    With
    \[
        \frac{1}{2}\scpr{v,v}+\scpr{v,p(0) \Lambda}+\frac{1}{2}|p(0)|_H^2\scpr{\Lambda,\Lambda}
        =\frac{1}{2}\scpr{v+p(0)\Lambda,v+p(0)\Lambda}
        =\frac{1}{2}\scpr{\tilde p,\tilde p}
    \]
    it follows
    \begin{align*}
     \scpr{p',\tilde p}+\scp{p(0),\tilde p(0)}
        &=\frac{1}{2}\left(
                       \e^{-2\rho}|p(1)|_H^2+
                       \scpr{\tilde p,\tilde p}
                     \right)\\
          &\quad +\rho\left(
                   \scpr{\m v,v}+
                   2\scpr{v,p(0)}+
                  |p(0)|^2_H \scpr{\Lambda,2-\m\Lambda}
               \right).
    \end{align*}
    Finally,
    \begin{align*}
        \scpr{\m v,v}+
        2\scpr{v,p(0)}+
        |p(0)|^2\scpr{\Lambda,1}
        &=\sum_{i=0}^q w_i\frac{1}{t_i}(t_i^2|v(t_i)|_H^2+2t_i\scp{v(t_i),p(0)}+|p(0)|_H^2) \\
        &=\sum_{i=0}^q w_i\frac{1}{t_i}|p(t_i)|^2_H \\
        &=\sum_{i=0}^q w_i\scp{p(t_i),\tilde p(t_i)}\\
        &=\scpr{p,\tilde p}
    \end{align*}
    gives
    \begin{align*}
     \scpr{p',\tilde p}+\scp{p(0),\tilde p(0)}
        &=\frac{1}{2}\left(
                       \e^{-2\rho}|p(1)|_H^2+
                       \scpr{\tilde p,\tilde p}
                     \right)
          +\rho\left(
                   \scpr{p,\tilde p}+
                   |p(0)|_H^2\scpr{\Lambda,1-\m\Lambda}
               \right).
    \end{align*}
    Using $\scpr{\Lambda,1-\m\Lambda}\geq 0$, which we provide in Lemma \ref{lem:auxLambda}, the first result is proved.
    The second one follows upon using the exactness of the quadrature rule and $t_i^{-1}>1$.
   \end{proof}

   \begin{lem}\label{lem:auxLambda}
    Let $\Lambda\in \PS_{q}[0,1]$ such that $\Lambda(t_i)=\frac{1}{t_i}$ for $i\in\{0,\ldots,q\}$, where $t_i$ is chosen as in Lemma \ref{lem:akmak}. Then
    \[
        \scpr{\Lambda,1-\m\Lambda}\geq 0,
    \]
    where $(\m\Lambda)(t)\coloneqq t\Lambda(t)$.
   \end{lem}
   \begin{proof}
    We rewrite the scalar product as a quadrature error:
    \[
        \scpr{\Lambda,1-\m\Lambda}
        =\sum_{i=0}^q w_i\frac{1}{t_i}-\int_0^1\e^{-2\rho t}t\Lambda^2(t)\dt
        =Q[f]-I[f]
    \]
    for $f$ given by $f(t)=t\Lambda^2(t)$, where $Q[g]=\sum_{i=0}^q w_i g(t_i)$ and $I[g]=\int_0^1 g$ for suitable $g$. There exists a constant $\alpha\in\R$ and a polynomial $w_0\in\PS_{q-1}[0,1]$, such that
    \[
        \Lambda(t)=\alpha t^{q}+w_0(t)
        \quad\text{which implies}\quad
        f(t)=\alpha^2 t^{2q+1}+w_1(t),
    \]
    where $w_1\in\PS_{2q}[0,1]$. Thus, setting $g(t)=t^{2q+1}$, we have that
    \[
        \scpr{\Lambda,1-\m\Lambda}
        =\alpha^2\left( Q[g]-I[g]\right),
    \]
    due to the exactness of the quadrature rule for polynomials of degree $2q$.

    Let $\Pi w\in \PS_{2q}[0,1]$ be an Hermite-interpolant of a given function $w$ satisfying
    \begin{align*}
        \Pi w(t_i)&=w(t_i),\quad i\in\{0,\dots,q\},\\
        (\Pi w)'(t_i)&=w'(t_i),\quad i\in\{0,\dots,q-1\}.
    \end{align*}
    Then it follows 
    \[
        Q[g]
           =\sum_{i=0}^q w_i t_i^{2q+1}
           =\sum_{i=0}^q w_i (\Pi g)(t_i^{2q+1})
           =Q[\Pi g]
           =I[\Pi g].
    \]
    Using that for each $t\in[0,1]$ there is $\zeta\in (0,1)$ such that
    \[
     (\Pi g)(t)- g(t)=-\frac{g^{(2q+1)}(\zeta)}{(2q+1)!}(t-1)\prod_{i=0}^{q-1}(t-t_i)^2= (1-t)\prod_{i=0}^{q-1}(t-t_i)^2
    \]see, for instance, \cite[Section 2.1.5]{SB02},
we infer that
    \begin{align*}
       \scpr{\Lambda,1-\m \Lambda}
        &=\alpha^2I[\Pi g- g]\\
        &=\alpha^2\int_0^1\e^{-2\rho t}\left(\prod_{i=0}^{q-1}(t-t_i)^2 \right)(1-t)\dt
         \geq 0.\qedhere
    \end{align*}
   \end{proof}

 Now we are able to improve Theorem \ref{thm:err2} following \cite[Corollary 2.1]{AM04} and \cite{VR13}.

   \begin{thm}\label{thm:discr_error}
    There exists $C\geq 0$ depending on $T$, $q$, $\|M_0\|$, $\|M_1\|$, $\gamma$, and $\rho$ such that
    \begin{align*}
        \sup_{t\in[0,T]}\scp{M_0\xi(t),\xi(t)}
        \leq Cg(U),
    \end{align*}
    with $g(U)$ defined as in Theorem~\ref{thm:err2}.
   \end{thm}
   \begin{proof}
       For the discrete error $\xi=U^\tau-PU\in \mathcal{U}^\tau$ we define $\phi$ by
       \[
           \phi|_{I_m}=P\left(t\mapsto \frac{\tau_m}{t-t_{m-1}}\xi(t) \right)\quad(m\in\{1,\ldots,M\}).
       \]
       Then for all $m\in\{1,\ldots,M\}$ and $i\in\{0,\ldots,q\}$ we have
       \[
           \scp{M_0\phi(\tmi),\phi(\tmi)}
              =\frac{\tau_m}{\tmi-t_{m-1}}\scp{M_0\xi(\tmi),\xi(\tmi)}
              \geq \scp{ M_0\xi(\tmi),\xi(\tmi)}
       \]
       and by Lemma~\ref{lem:akmak} (apply the lemma to the functions $p=\sqrt{M_0}\xi$ and $\tilde p=\sqrt{M_0}\phi$ rescaled on $[0,1]$)
       \begin{align*}
          \Qmr{\partial_t M_0\xi,2\phi}+\scp{M_0\xi_{m-1}^+,2\phi_{m-1}^+}
               &\geq \frac{1}{\tau_m}\Qmr{M_0 \phi,\phi}
                \geq \frac{1}{\tau_m}\Qmr{M_0 \xi,\xi}.
       \end{align*}
        By the equivalence of norms on $\mathcal{P}_q([0,1])$, there exists $K_1\geq 0$ depending on $q$ only, such that
       \[
        \sup_{t\in[0,1]}|p(t)| \leq K_1 \intop_0^1 |p(t)| \dd t \quad(p\in \mathcal{P}_q([0,1])).
       \]
  Consequently, we obtain for all $m\in\{1,\ldots,M\}$
       \[
           \sup_{t\in I_m}\scp{M_0\xi(t),\xi(t)}
              \leq \frac{K_1}{\tau_m}\e^{2\rho\tau_m}\Qmr{M_0\xi,\xi}
              \leq \frac{K}{\tau_m}\Qmr{M_0\xi,\xi}
       \]
       where $K:=K_1e^{2\rho T}\geq\max\limits_{m\in\{1,\dots,M\}}\{\e^{2\rho\tau_m}\}\geq K_1$. Moreover, we have
       \begin{align*}
           \Qmr{A\xi,2\phi}
              &=\frac{\tau_m}{2}\sum_{i=0}^q\omega^m_i
                       \scp{A\xi(\tmi),2\phi(\tmi)}\\
              &=\frac{\tau_m}{2}\sum_{i=0}^q\omega^m_i\frac{2\tau_m}{\tmi-t_{m-1}}
                       \scp{A\xi(\tmi),\xi(\tmi)}
               =0
       \end{align*}
       upon $A=-A^*$. Together, it follows for all $m\in\{1,\ldots,M\}$
       \begin{align*}
           &\sup_{t\in I_m}\scp{M_0\xi(t),\xi(t)}\\
           &  \leq K\left( \Qmr{(\partial_t M_0+M_1+A)\xi,2\phi}+\scp{M_0\xi_{m-1}^+,2\phi_{m-1}^+}-\Qmr{M_1 \xi ,2\varphi}\right) \\
            & = K\bigg( \Qmr{(\partial_t M_0+M_1+A)\xi,2\phi}
                           +\scp{M_0\jump{\xi}^0_{m-1},2\phi_{m-1}^+}\\
                           &\quad \quad \quad +\scp{M_0\xi^-_{m-1},2\phi_{m-1}^+}-\Qmr{M_1 \xi,2\varphi}
                     \bigg).
       \end{align*}
       Using the error equation \eqref{eq:error2_} with $\chi=2\phi$ (recall $\eta=U-PU$), we obtain
       \begin{align*}
           \sup_{t\in I_m}\scp{M_0\xi(t),\xi(t)}
              &\leq K\bigg( \Qmr{(\partial_t M_0+M_1+A)\eta,2\phi}
                     +\scp{M_0\jump{\eta}_{m-1},2\phi_{m-1}^+}\\
            &\quad \quad \quad          +\scp{M_0\xi^-_{m-1},2\phi_{m-1}^+}-\Qmr{M_1 \xi,2\varphi}\bigg).
       \end{align*}
       Using Lemma~\ref{lem:rhs2_1}, Lemma \ref{lem:rhs2_2} with $\psi=2\phi$ and Theorem~\ref{thm:err2}, we estimate further with some $C_1\geq1$ depending on $q$, $T$, and $\rho$ such that
       \begin{align*}
           \sup_{t\in I_m}\scp{M_0\xi(t),\xi(t)}&\leq K\bigg(
                        \Qmr{\partial_t M_0(U-\widehat{P} U),2\phi}
                       + R(U,2\phi)\\
                                     &\quad \quad \quad        +\scp{M_0\xi^-_{m-1},2\phi_{m-1}^+}-\Qmr{M_1 \xi,2\varphi}
                     \bigg)\\
              &\leq C_1\alpha_1\bigg(
                                 \tnorm{\partial_t M_0(U-\widehat{P} U)}_{Q,\rho,m}^2
                                +|M_0(P U-U)(t_{m-1}^+)|^2                                
                             \bigg)\\&\quad \quad \quad        
                   +C_1\alpha_2\scp{M_0\xi^-_{m-1},\xi^-_{m-1}}+C_1\alpha_1\|M_1\|^2 \tnorm{\xi}_{Q,\rho,m}^2                    \\&\quad \quad \quad        +3\beta_1\Qmr{2\phi,2\phi}
                   +\beta_2\scp{2M_0\phi^+_{m-1},2\phi_{m-1}^+},                  
       \end{align*}
       where $\alpha_i\beta_i=\frac{1}{4}$, $i\in\{1,2\}$ and we used that \[\langle M_0 u,v\rangle=\langle \sqrt{M_0} u,\sqrt{M_0}v\rangle\leq  \langle M_0 u,u\rangle\langle M_0 v,v\rangle\] for all $u,v\in H$, by the non-negativity and selfadjointness of $M_0$. Using Theorem~\ref{thm:err2} (and Remark~\ref{r:err2}), we, thus, get 
       \begin{equation}\label{e:at}
          \sup_{t\in I_m}\scp{M_0\xi(t),\xi(t)}\leq C(\alpha_1+\alpha_2) g(U)+3\beta_1\Qmr{2\phi,2\phi}                   +\beta_2\scp{2M_0\phi^+_{m-1},2\phi_{m-1}^+}
       \end{equation}
       for some $C\geq 1$ depending on $q$, $T$, $\rho$, and $\|M_1\|$, where $g(U)$ is defined in Theorem \ref{thm:err2}.
      Next, by Corollary \ref{c:t0}, we find $c>0$ depending on $\rho$ and $T$ only such that
      \[
         \frac{\tau_m}{t_{m,i}-t_{m-1}}\leq \frac{\tau_m}{t_{m,0}-t_{m-1}} \leq \frac{1}{c} \quad(m\in\{1,\ldots,M\}).
      \]
      Hence, for all $m\in\{1,\ldots,M\}$,
       \begin{multline*}
           \Qmr{2\phi,2\phi}
            \leq L\Qmr{\xi,\xi}
            = (4/c)\tnorm{\xi}_{Q,\rho,m}^2
           \\ \text{and}\quad
           \scp{2M_0\phi^+_{m-1},2\phi_{m-1}^+}
              \leq (4/c)\sup_{t\in I_m}\scp{M_0\xi(t),\xi(t)}.
       \end{multline*}
       Next, we choose $\beta_2=(4/c)\frac{1}{2}$. Thus, appealing to \eqref{e:at}, we obtain for all $m\in\{1,\ldots,M\}$
       \begin{align*}
             \frac{1}{2}\sup_{t\in I_m}\scp{M_0\xi(t),\xi(t)} & = \sup_{t\in I_m}\scp{M_0\xi(t),\xi(t)} - \frac{1}{2}\sup_{t\in I_m}\scp{M_0\xi(t),\xi(t)}
             \\ & \leq C(\alpha_1+\alpha_2) g(U)+3\beta_1(4/c)\tnorm{\xi}_{Q,\rho,m}^2,
       \end{align*}
       using Theorem \ref{thm:err2} (i.e.~Remark \ref{r:err2}) again for the second term on the right-hand side und computing the supremum over $m\in\{1,\ldots,M\}$ in the latter inequality, we obtain the assertion.
   \end{proof}
   \subsection{Estimating the interpolation error in time}\label{su:ip}

   In the previous section we showed that the discrete error is bounded in terms of the interpolation errors.
   We finalize the error estimates in time in this section focussing on the interpolation error. The aim and, thus, main theorem of this section is Theorem \ref{thm:err_inter}, where we estimate the difference between the exact solution $U$ of \eqref{e:op} and the solution $U^\tau$ of the quadrature formulation \eqref{eq:discr_quad_form} with right-hand side $PF$ and initial value $U(0+)$. We use the same notation as in the previous section. In addition, we set $\tau\coloneqq \max\{\tau_m:m\in\{1,\ldots,M\}\}$. Moreover, shall further assume that the hypotheses of Theorem \ref{thm:Pic} are in effect.

   \begin{lem}\label{lem:dthatP}
      There exists $C\geq 0$ depending on $q$ and $T$ such that for all $V\in H_\rho^{q+3}(\mathbb{R};H)$
      \[
         \tnorm{\partial_t (V-\widehat{P} V)}_{Q,\rho}
            \leq C\tau^{q+1} |\partial_t^{q+3}V|_\rho.
      \]      
   \end{lem}
   \begin{proof}
      First we note that $H_\rho^{q+3}(\mathbb{R};H)\hookrightarrow C_\rho^{q+2}(\mathbb{R};H)$ by the Sobolev-embedding theorem. By the definition of $\tnorm{\cdot}_{Q,\rho}$
      we have that
      \begin{align*}
         \tnorm{\partial_t (V-\widehat{P} V)}_{Q,\rho}^2
            &= \sum_{m=1}^M\Qm{|\partial_t(V-\widehat{P} V)|_H^2\e^{-2\rho t}}\\
            &= \sum_{m=1}^M \frac{\tau_m}{2} \sum_{i=0}^q \omega^m_i |(\partial_t(V-\widehat{P} V))(\tmi)|_H^2\e^{-2\rho \tmi}.
      \end{align*}
      Using the standard result from interpolation theory
      \[
         \sup_{t\in I_m} |(v-\widehat{P} v)'(t)|
            \leq C \tau_m^{q+1} \sup_{t\in I_m} |v^{(q+2)}(t)|,
      \]
      for all $v \in W^{q+2,\infty}(0,T)$ we obtain
      \begin{align*}
         \tnorm{\partial_t (V-\widehat{P} V)}_{Q,\rho}^2
            &\leq C^2\sum_{m=1}^M \frac{\tau_m}{2}
                   \tau_m^{2(q+1)}
                   \sum_{i=0}^q \omega^m_i
                      \sup_{t\in I_m}
                         |\partial_t^{p+2}V(t)|_H^2
                   \e^{-2\rho \tmi}\\
            &\leq C^2\tau^{2(q+1)}
                   \sup_{t\in [0,T]}
                         |\partial_t^{p+2}V(t)|_H^2.\qedhere
      \end{align*}
   \end{proof}
  For the next two lemmas, we recall the standard result from interpolation theory
      \begin{gather}\label{eq:inter_standard}
         \sup_{t\in I_m} |(v-P v)(t)|
            \leq C \tau_m^{q+1} \sup_{t\in I_m} |v^{(q+1)}(t)|,
      \end{gather}
      for all $v \in W^{q+1,\infty}(0,T)$, see, for instance, \cite[Section 2.1.4]{SB02}.
   \begin{lem}\label{lem:jump}
     There exists $C\geq 0$ depending on $q$ and $T$ such that for all $V\in H_\rho^{q+2}(\mathbb{R};H)$
      \[
         |(V-P V)(t_{m-1}^+)|_H
            \leq C\tau_m^{q+1} |\partial_t^{q+2}V|_\rho.                     
      \]      
   \end{lem}
   \begin{proof}
      We obtain
      \begin{align*}
        |(V-PV)(t_{m-1}^+)|_H
            &\leq \sup_{t\in I_m}
                     |(V-PV)(t)|_H\\
            &\leq C\tau_m^{q+1}
                     \sup_{t\in I_m}
                        |\partial_t^{q+1}V(t)|_H.
      \end{align*}
      The claim follows from the Sobolev-embedding theorem.
   \end{proof}

   With the previous lemmas we can already estimate $g(U)$. Now let us estimate the final term needed to
   estimate the error $U-U^\tau$.

   \begin{lem}\label{lem:inter_scp} There exists $C\geq 0$ depending on $T$ and $q$ such that
      for all $U\in H_\rho^{q+2}(\mathbb{R};H)$
      \begin{align*}
         \sup_{t\in [0,T]}&\scp{M_0(U-P U)(t),(U-P U)(t)}\\
            &\leq C\tau^{2(q+1)}|\partial_t^{q+2} U|^2_\rho.
      \end{align*}
   \end{lem}
   \begin{proof}
      Using the Cauchy--Schwarz and Young inequality we derive
      \begin{align*}
         \sup_{t\in [0,T]}&\scp{M_0(U-P U)(t),(U-P U)(t)}\\
            &\leq \frac{1}{2}
                  \left(
                     \sup_{t\in [0,T]}|M_0(U-P U)(t)|_H^2
                    +\sup_{t\in [0,T]}|(U-P U)(t)|_H^2
                  \right).
      \end{align*}
      Using \eqref{eq:inter_standard} with $v=M_0U$ and $v=U$ we obtain
      \begin{align*}
         \sup_{t\in [0,T]}|M_0(U-P U)(t)|_H
            &\leq C \tau^{q+1}\sup_{t\in [0,T]}|\partial_t^{q+1}M_0 U(t)|_H,\\
         \sup_{t\in [0,T]}|(U-P U)(t)|_H
            &\leq C \tau^{q+1}\sup_{t\in [0,T]}|\partial_t^{q+1}U(t)|_H.
      \end{align*}
      Combining these results the claim follows from the Sobolev-embedding theorem.
   \end{proof}

   Combining the previous lemmas, Theorem \ref{thm:err2} and Theorem \ref{thm:discr_error}, we can bound the discrete error in time.
   \begin{thm}\label{thm:err_inter}
      Assume that $U\in H_\rho^{q+3}(\mathbb{R};H)$. Then there exists $C\geq 0$ depending on $\|M_0\|$, $\|M_1\|$, $\rho$, $T$, $\gamma$, $q$ such that
      \[
          \sup_{t\in[0,T]}\scp{M_0(U-U^\tau)(t),(U-U^\tau)(t)}+
          \tnorm{U-U^\tau}_{Q,\rho}^2
          \leq C \tau^{2(q+1)}|\partial_t^{q+3}U|^2_\rho.
      \]
   \end{thm}
   
   \begin{proof}
    By Lemma \ref{lem:dthatP} and Lemma \ref{lem:jump} applied to $V=M_0U$ we have that
    \[
    g(U)\leq C_1\tau^{2(q+1)} |\partial_t^{q+3}U|^2_\rho
    \]for some $C_1\geq0$.
We note that $\tnorm{U-U^\tau}_{Q,\rho}\leq \tnorm{\eta}_{Q,\rho}+\tnorm{\xi}_{Q,\rho}=\tnorm{\xi}_{Q,\rho}$ and hence, by Theorem \ref{thm:err2} we obtain
    \[
      \tnorm{U-U^\tau}_{Q,\rho}^2\leq g(U).
    \]
    Moreover, 
    \[
    \scp{M_0(U-U^\tau)(t),(U-U^\tau)(t)}\leq 2\scp{M_0\eta(t),\eta(t)}+2\scp{M_0\xi(t),\xi(t)}
    \]
and thus, by Theorem \ref{thm:discr_error} and Lemma \ref{lem:inter_scp} we infer
\[
\sup_{t\in[0,T]}\scp{M_0(U-U^\tau)(t),(U-U^\tau)(t)}\leq C( \tau^{2(q+1)}|\partial_t^{q+2} U|^2_\rho+ g(U))
\]
for some $C\geq0$. Combining these estimates, the claim follows.  
   \end{proof}

   \begin{rem}
      The above analysis holds for all evolutionary problems and gives error bounds for the
      semi-discrete solution of order $q+1$, assuming enough regularity in time.
      For a fully discrete method, a spatial discretisation
      has to be defined too. This step, however, has to be done for each problem considered separately.
   \end{rem}

   \section{Full discretisation for Example~\ref{ex:CTS}}\label{s:Full}
      Let us assume a regular, quasi uniform and shape-regular triangulation $\Omega_h$
      of $\Omega$ into triangular open cells $\sigma$ with maximal cell diameter $h$.
      Moreover, we assume that the interfaces between $\Omega_{\mathrm{ell}}$, $\Omega_{\mathrm{par}}$ and $\Omega_{\mathrm{hyp}}$ are polygonal such that the triangulation $\Omega_h$ fits to these interfaces.

      As the whole article is mainly concerned with the correct time-discretization, in this section, we will employ the custom of the ``generic constant'' $C\geq0$ that may vary from line to line, which, however, depends on $T$, $\rho$, $\|M_1\|$, $\|M_0\|$, $q$, and $\gamma$ and on $k$, the order of the assumed spatial regularity, only.
      
      Then the fully discretised counterpart $\U^\tau_h$ to $\U$ is given by
      \[
          \U^\tau_h := \left\{
                      (u_h,v_h)\in\U^\tau:\,
                      u_h|_{I_m}\in\PS_q(I_m,V_1(\Omega)),\,
                      v_h|_{I_m}\in\PS_q(I_m,V_2(\Omega)),\,
                      m\in\{1,\dots,m\}
                  \right\},
      \]
      where the spatial spaces are
      \begin{align*}
        V_1(\Omega)&:=\left\{v\in H_0^1(\Omega);\forall \sigma:\,v|_\sigma\in\PS_k(\sigma)\right\},\\
        V_2(\Omega)&:=\left\{w\in H(\Div,\Omega);\forall \sigma:\,w|_\sigma\in RT_{k-1}(\sigma)\right\}.
      \end{align*}
      Here $\PS_k(\sigma)$ is the space of polynomials of degree up to $k$ on the cell
      $\sigma$ and $RT_{k-1}(\sigma)$ is a the Raviart-Thomas-space, defined by
      \[
         RT_{k-1}(\sigma)=(\PS_{k-1}(\sigma))^n+\vx\PS_{k-1}(\sigma).
      \]
      Note that
      \begin{align*}
          \PS_{k-1}(\sigma)\subset RT_{k-1}(\sigma)&\subset\PS_k(\sigma),\\
          \Div(RT_{k-1}(\sigma))&\subset \PS_{k-1}(\sigma)
          \quad\text{and}\quad \\
          RT_{k-1}(\sigma)\cdot\vn|_{\partial\sigma}&\subset \PS_{k-1}(\partial\sigma).
      \end{align*}
      
        Furthermore, if the mesh consists of quadrilateral or hexahedral cells, in above definitions 
  and statements the polynomials space $\PS_k(\sigma)$ can be replaced by $\QS_k(\sigma)$ including all 
  polynomials of total degree $k$ over $\sigma$.
      
      \begin{rem}[Solvability of the fully discrete system]
        We can apply the general existence theory that was also used in Proposition~\ref{prop:solvability}. More precisely, the positive definiteness still holds, since the triangulation fits to the interfaces and hence, the uniqueness of the system is warranted. However, since the problem is finite-dimensional, the uniqueness implies the existence of a solution of the fully discretised problem.
      \end{rem}
      Let us come to the interpolation operator $I=(I_1,I_2)$.
      For $I_1:C(\Omega)\to V_1$ we use the Scott--Zhang interpolant on each cell $\sigma$, 
  see \cite{SZ90} for a precise definition, that is patched together continuously.
  Here local interpolation error estimates can be given 
  using $L^2$-norms also in 3d, which is not possible for standard Lagrange 
  interpolation.
      For $I_2:W\to V_2$ with \[W(\sigma)=\left\{q\in (L^s(\sigma))^n:\,\Div q\in L^2(\sigma)\right\},\; s>2,\] we also
      use the standard interpolator, defined via moments, see \cite{BF91}. Note that in the following, in order to avoid a cluttered notation as much as possible, we will not explicitly keep track on the number of components of the $L^2(\Omega)$- or $H^k(\Omega)$-spaces under consideration, as it will be obvious from the context.

      Standard local interpolation error estimates yield for all $v\in H_0^1(\Omega)\cap H^r(\Omega)$
      \begin{align*}
        \norm{v-I_1 v}{0,\Omega}       &\leq C h^r \norm{v}{r,\Omega},\\
        \norm{\grad(v-I_1v)}{0,\Omega} &\leq C h^{r-1} \norm{v}{r,\Omega},
      \end{align*}
      where $1\leq r \leq k+1$ and for all $q\in H^s(\Omega)$ such that $\Div q\in H^s(\Omega)$
      \begin{align*}
        \norm{q-I_2 q}{0,\Omega}       &\leq C h^s \norm{q}{s,\Omega},\\
        \norm{\Div(q-I_2 q)}{0,\Omega} &\leq C h^s \norm{\Div q}{s,\Omega},
      \end{align*}
      where $1\leq s\leq k$, see \cite{BF91}.

      Let $U_h^\tau\in\U^\tau_h$ be the solution of the fully discretised system and $PIU\in\U^\tau_h$ be the interpolated solution of \eqref{eq:evo}
      for the operators $M_0,M_1$ given in Example~\ref{ex:CTS} and $A$ given as in Example \ref{ex:illus}. Then we obtain analogously to the derivation of the errors of the semi-discretisation
      \begin{multline}\label{eq:full_estimate}
          \sup_{t\in[0,T]}\scp{M_0(PIU-U_h^\tau)(t),(PIU-U_h^\tau)(t)}+
          \tnorm{PIU-U_h^\tau}_{Q,\rho}^2\\
          \leq C\Bigg(
                    \tnorm{\partial_t M_0(U-\widehat{P}IU)}_{Q,\rho}^2
                   +\tnorm{M_1(U-PIU)}_{Q,\rho}^2
                   +\tnorm{A  (U-PIU)}_{Q,\rho}^2\\
                   +T\max_{1\leq m\leq M}\left\{
                                            |M_0(PIU-IU)(t_{m-1}^+)|_H^2\e^{-2\rho t_{m-1}}
                                         \right\}
                 \Bigg),
      \end{multline}
      where we remark that in contrast to Theorem \ref{thm:err2}  the terms $\tnorm{M_1(U-PIU)}_{Q,\rho}^2$ and $\tnorm{A  (U-PIU)}_{Q,\rho}^2$ do not vanish, since we also interpolate with respect to space.
      In the following group of lemmas we estimate the terms on the right-hand side of \eqref{eq:full_estimate} 
      and start with a term partcularly needed for the final convergence estimate in Theorem \ref{thm:cts}. Beforehand, let us introduce 
      \[
          \norm{u}{Q,\rho,k,D}^2
               = \sum_{m=1}^M\Qm{|u|^2_{k,D}}\e^{-2\rho t_{m-1}}
      \]
      where $D\subseteq \Omega$ is measurable.
      \begin{lem}\label{lem:full_interp1}
         It holds for $U=(u,v)\in H_\rho(\mathbb{R};H^{k}(\Omega)\times H^k(\Omega))$ 
         \[
              \tnorm{U-PIU}_{Q,\rho}
                 \leq C h^{k}\left( 
                           \norm{u}{Q,\rho,k,\Omega}+
                            h^{k}\norm{v}{Q,\rho,k,\Omega}
                        \right),
           \]
           Moreover, if $U=(u,v)\in H_\rho(\mathbb{R};D(A))$ such that $AU\in H_\rho(\mathbb{R};H^{k}(\Omega)\times H^k(\Omega))$, then
              \[\tnorm{A(U-PIU)}_{Q,\rho}
                 \leq Ch^{k}
                      \left( 
                            \norm{u}{Q,\rho,k+1,\Omega}+
                            \norm{\Div v}{Q,\rho,k,\Omega}
                      \right).           
          \]
      \end{lem}
      \begin{proof}
         By the definition of $\Qmr{\cdot}$ we have
         \begin{align*}
             \tnorm{U-PIU}_{Q,\rho}^2
                &=
             \tnorm{U-IU}_{Q,\rho}^2\\
                &=\sum_{m=1}^M\Qm{\norm{U-IU}{0,\Omega}^2}\e^{-2\rho t_{m-1}}\\
                &=\sum_{m=1}^M\Qm{\norm{u-I_1 u}{0,\Omega}^2+
                                  \norm{v-I_2 v}{0,\Omega}^2}\e^{-2\rho t_{m-1}}\\
                &\leq C\sum_{m=1}^M\Qm{ h^{2(k+1)}|u(\cdot)|_{k+1,\Omega}^2+
                                        h^{2k}|v(\cdot)|_{k,\Omega}^2}\e^{-2\rho t_{m-1}}\\
                &= C \left( 
                            h^{2(k+1)}\norm{u}{Q,\rho,k+1,\Omega}^2+
                            h^{2k}\norm{v}{Q,\rho,k,\Omega}^2
                        \right).
         \end{align*}
         Very similarly we have for the second norm
         \begin{align*}
             \tnorm{A(U-PIU)}_{Q,\rho}^2
                &= \tnorm{A(U-IU)}_{Q,\rho}^2\\
                &=\sum_{m=1}^M\Qm{\norm{\grad(u-I_1 u)}{0,\Omega}^2+
                                  \norm{\Div(v-I_2 v)}{0,\Omega}^2}\e^{-2\rho t_{m-1}}\\
                &\leq C\sum_{m=1}^M\Qm{ h^{2k}|u|_{k+1,\Omega}^2+
                                        h^{2k}|\Div v|_{k,\Omega}^2}\e^{-2\rho t_{m-1}}\\
                &= Ch^{2k}\left( 
                               \norm{u}{Q,\rho,k+1,\Omega}^2+
                               \norm{\Div v}{Q,\rho,k,\Omega}^2
                          \right).\qedhere
         \end{align*}       
      \end{proof}
      
      \begin{lem}\label{lem:full_interp2}
         It holds for $U=(u,v)\in H_\rho(\mathbb{R};H^{k}(\Omega)\times H^k(\Omega))$
     \[
            \tnorm{M_1(U-PIU)}_{Q,\rho}^2
                 \leq C h^{2k}\left(
                                       \norm{u}{Q,\rho,k,\Omega}+
                                       \norm{v}{Q,\rho,k,\Omega}
                                    \right)^2
       \]
      \end{lem}
      \begin{proof}
         The assertion follows from Lemma \ref{lem:full_interp1} and the boundedness of $M_1$.
      \end{proof}
      
      \begin{lem}\label{lem:full_interp3}
         For $U=(u,v)\in H_\rho^1(\mathbb{R};H^{k}(\Omega)\times H^k(\Omega))\cap H_\rho^{q+2}(\mathbb{R};L^2(\Omega)\times L^2(\Omega))$ we have that
          \begin{multline*}
             \sup_{t\in[0,T]}\scp{M_0(U-PIU)(t),(U-PIU)(t)}\\
                 \leq C \Bigg(
                              h^{2k}\sup_{t\in[0,T]}
                                    \left(
                                          \snorm{u(t)}{k,\Omega}+
                                          \snorm{v(t)}{k,\Omega}
                                    \right)^2                   +\tau^{2(q+1)}|\partial_t^{q+2} IU|_\rho       
                        \Bigg).
          \end{multline*}
          
      \end{lem}
      \begin{proof}
         The operator $M_0$ is selfadjoint and non-negative.          Thus it follows that
         \begin{align*}
             &\langle M_0(U-PIU)(t),(U-PIU)(t)\rangle_{L^2(\Omega)}\\
             &=|\sqrt{M_0}(U-PIU)(t)|_{L^2(\Omega)}^2\\
             &\leq 2 \left(|\sqrt{M_0}(U-IU)(t)|^2_{L^2(\Omega)}+|\sqrt{M_0}(IU-PIU)(t)|^2_{L^2(\Omega)} \right)
         \end{align*}
         for each $t\in [0,T]$. The second term can be estimated by
         \[
         |\sqrt{M_0}(IU-PIU)(t)|^2_{L^2(\Omega)}\leq C\tau^{2(q+1)}|\partial_t^{q+2}IU|_\rho^2
         \]
according to Lemma \ref{lem:inter_scp}, while the first term can be estimated by
\[
|\sqrt{M_0}(U-IU)(t)|^2_{L^2(\Omega)}\leq C h^{2k}\left(|u(t)|^2_{k,\Omega}+|v(t)|^2_{k,\Omega}\right),
\]
due to the boundedness of $\sqrt{M_0}$. Hence, the assertion follows.
      \end{proof}

      \begin{lem}\label{lem:full_interp4}
          For $U=(u,v)\in H_\rho^1(\mathbb{R};H^{k}(\Omega)\times H^k(\Omega))\cap H_\rho^{q+3}(\mathbb{R};L^2(\Omega)\times L^2(\Omega))$ we get
\[
              \tnorm{\partial_t M_0(U-\widehat{P}IU)}_{Q,\rho}^2
                  \leq C\Bigg(h^{2k}\left(
                                          \norm{\partial_t u}{Q,\rho,k,\Omega}+
                                          \norm{\partial_t v}{Q,\rho,k,\Omega}
                                       \right)^2
                                  +\tau^{2(q+1)}|\partial_t^{q+3}IU|_\rho
                       \Bigg).
  \]
          \end{lem}
      \begin{proof}
         We have that
         \begin{align*}
            \tnorm{\partial_t M_0(U-\widehat{P}IU)}_{Q,\rho}
               &\leq  \tnorm{\partial_t M_0(U-IU)}_{Q,\rho}
                     +\tnorm{\partial_t (M_0IU-\widehat{P} M_0IU)}_{Q,\rho}\\
               &\leq \tnorm{M_0(\partial_t U-I\partial_t U)}_{Q,\rho}
                     +C\tau^{q+1}|\partial_t^{q+3}IU|_\rho,
         \end{align*}
         by Lemma~\ref{lem:dthatP}. For the first term we have by Lemma \ref{lem:full_interp1}
         \begin{equation*}
              \tnorm{M_0(\partial_t U-I\partial_t U)}_{Q,\rho}^2
                 \leq C  h^{2k}
                                 \left(
                                    \norm{\partial_t u}{Q,\rho,k,\Omega}^2+
                                    \norm{\partial_t v}{Q,\rho,k,\Omega}^2
                                 \right).
                        \qedhere
         \end{equation*}       
      \end{proof}

      \begin{lem}\label{lem:full_interp5}
         It holds for $U=(u,v)\in H_\rho^{q+2}(\mathbb{R};L^2(\Omega)\times L^2 (\Omega))$
         \[
              \max_{1\leq m\leq M}\left\{
                                      |M_0(PIU-IU)(t_{m-1}^+)|_{L^2(\Omega)}\e^{-\rho t_{m-1}}
                                  \right\}
               \leq C\tau^{q+1}
                                  |\partial_t^{q+2}IU|_\rho.                        
          \]
          
      \end{lem}
      \begin{proof}
         This is a direct consequence of Lemma~\ref{lem:jump}.
      \end{proof}

      Lemma~\ref{lem:full_interp1} to \ref{lem:full_interp5} give us all needed estimates 
      for the final convergence result for Example~\ref{ex:CTS}.

      \begin{thm}\label{thm:cts}
         We assume for the solution $U=(u,v)$ of Example \ref{ex:CTS} the regularity 
         
         \[
         U\in H_\rho^{1}(\mathbb{R};H^k(\Omega)\times H^k(\Omega))\cap H_\rho^{q+3}(\mathbb{R};L^2(\Omega)\times L^2(\Omega)) 
         \]
as well as 
\[
AU\in H_\rho(\mathbb{R}; H^k(\Omega)\times H^k(\Omega)).
\]
         Then we have for the error of the numerical solution by \eqref{eq:discr_quad_form}
         \[
             \sup_{t\in[0,T]}\scp{M_0(U-U_h^\tau)(t),(U-U_h^\tau)(t)}+
             \tnorm{U-U_h^\tau}_{Q,\rho}^2
             \leq C (\tau^{2(q+1)} + T h^{2k}).
         \]
      \end{thm}
  \section{Numerical examples}\label{s:ne}
  In the following section we consider some examples to verify numerically our
  theoretical findings.
  
  \subsection{Changing type system -- one space dimension}
  Let $\Omega=(-\frac{\pi}{2},\frac{\pi}{2})\subset\R^1$,
  $\Omega_h=(-\frac{\pi}{2},0)$ and $\Omega_p=(0,\frac{\pi}{2})$. 
  The problem is given on $\R\times\Omega$ by
  \begin{subequations}\label{eq:prob1}
  \begin{gather}\label{eq:prob1_pde}
     \left( 
      \partial_t
      \begin{pmatrix}
       1 & 0\\
       0 & \id{h}
      \end{pmatrix}
      +\begin{pmatrix}
        0 & 0\\
        0 & \id{p}
       \end{pmatrix}
      +\begin{pmatrix}
        0 & \partial_x\\
        \partial_x & 0
       \end{pmatrix}
     \right)\begin{pmatrix}
             u\\
             v
            \end{pmatrix}
       =\begin{pmatrix}
               f\\
               g
        \end{pmatrix}
  \end{gather}
  with $u(t,-\frac{\pi}{2})=u(t,\frac{\pi}{2})=0$, $g=0$ and
  \begin{gather}
   f(t,x)=\id{\R_{\geq 0}}(t)(2\e^t-1-t\id{\left( -\frac{\pi}{2},0\right)}(x))\cos(x).
  \end{gather}
  \end{subequations}
  The solution can be derived as
  \begin{align*}
   u(t,x) &= \id{\R_{\geq 0}}(t)(\e^t-1)\cos(x),\\
   v(t,x) &= \id{\R_{\geq 0}}(t)(\e^t-1-t\id{\left( -\frac{\pi}{2},0\right)}(x))\sin(x).
  \end{align*}
  Note that a priori, we impose no transmission condition. However, as in \cite[Remark 3.2]{W_StH16}, they can be derived for $u$ satisfying \eqref{eq:prob1} as
  \[
   u(t,0+) = u(t,0-),\quad
   \partial_x u(t,0+) = \int_0^t\partial_x u(s,0-)ds.
  \]
  The solution up to a time $T=1$ is shown in Figure~\ref{fig:prob1}.
  \begin{figure}[ht]
     \begin{center}
      \includegraphics[width=0.4\textwidth]{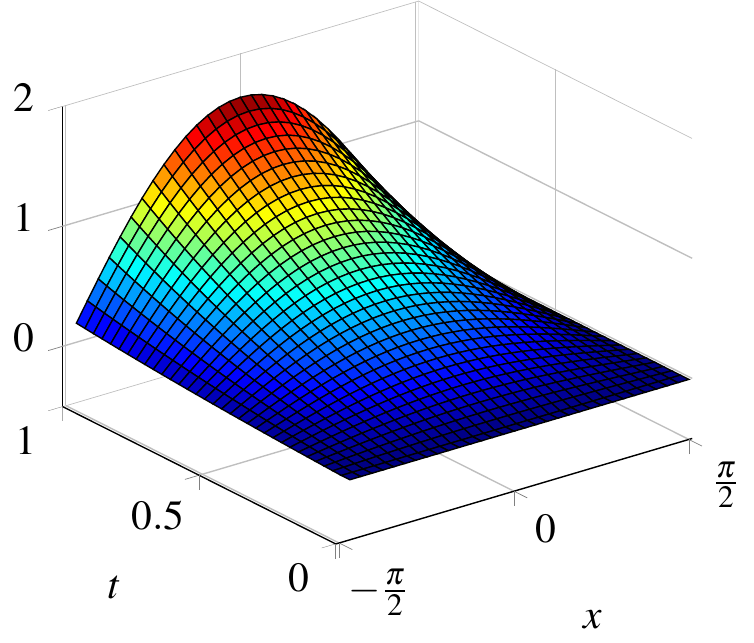}\quad
      \includegraphics[width=0.4\textwidth]{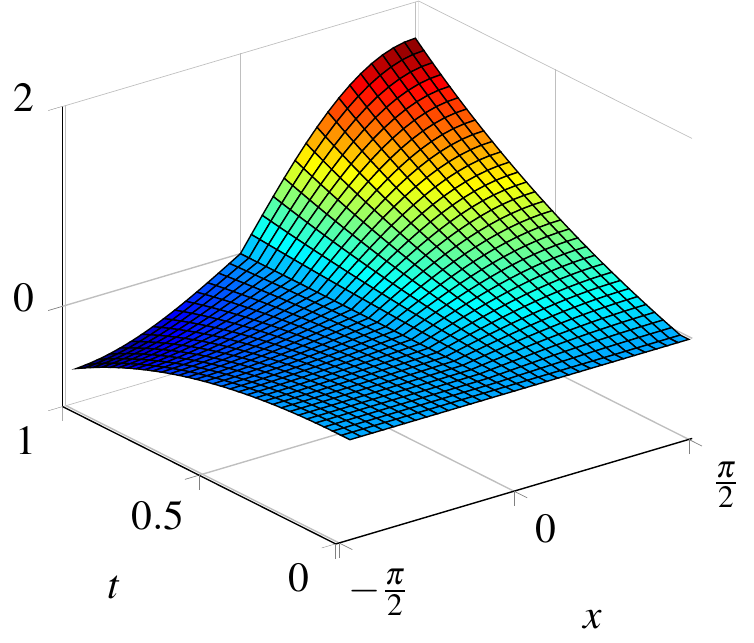}
     \end{center}
     \caption{Solution $u$ (left) and $v$ (right) of problem \eqref{eq:prob1}\label{fig:prob1}}
  \end{figure}
  For the numerical solution we use again $T=1$, an equidistant mesh of $M$ cells in time and an equidistant mesh
  of $N$ cells in space. In order to resolve the boundary $S=\overline\Omega_h\cap\overline\Omega_p=\{0\}$ we assume
  $N$ to be even. Note that we can use $\rho=1$ for the given solution $u$. 
  
  Defining
  \[
    E_{\sup}(v) = \left( \sup_{t\in[0,T]}\scp{M_0v(t),v(t)}\right)^{1/2},\quad
    E(v)   = \left( \sup_{t\in[0,T]}\scp{M_0v(t),v(t)}+\tnorm{v}_{Q,\rho}^2\right)^{1/2}
  \]
  we consider in Table~\ref{tab:prob1}
  \begin{table}[ht]
     \caption{Convergence results for $U-U_h$ of problem \eqref{eq:prob1}
              \label{tab:prob1}}
     \begin{center}   
      \begin{tabular}{rllllll}
         $N=M$ & 
         \multicolumn{2}{c}{$E_{\sup}(U-U_h)$} &
         \multicolumn{2}{c}{$\tnorm{U-U_h}_{Q,\rho}$} & 
         \multicolumn{2}{c}{$\tnorm{U-U_h}_{\rho}$}\\
         \hline
         \multicolumn{7}{c}{$p=2$, $q=1$}\\
         \hline
           8 & 8.727e-03 &      &  7.766e-04 &      &  1.855e-03 &     \\
          16 & 2.335e-03 & 1.90 &  1.939e-04 & 2.00 &  4.638e-04 & 2.00\\
          32 & 6.039e-04 & 1.95 &  4.851e-05 & 2.00 &  1.160e-04 & 2.00\\
          64 & 1.535e-04 & 1.98 &  1.213e-05 & 2.00 &  2.899e-05 & 2.00\\
         128 & 3.871e-05 & 1.99 &  3.032e-06 & 2.00 &  7.248e-06 & 2.00\\
         256 & 9.717e-06 & 1.99 &  7.580e-07 & 2.00 &  1.812e-06 & 2.00\\
         512 & 2.434e-06 & 2.00 &  1.895e-07 & 2.00 &  4.530e-07 & 2.00\\
         \hline
         \multicolumn{7}{c}{$p=3$, $q=2$}\\
         \hline
           8 & 6.963e-05 &      &  3.079e-06 &      &  1.717e-05 &     \\
          16 & 8.705e-06 & 3.00 &  1.898e-07 & 4.02 &  2.120e-06 & 3.02\\
          32 & 1.088e-06 & 3.00 &  1.182e-08 & 4.00 &  2.642e-07 & 3.00\\
          64 & 1.360e-07 & 3.00 &  7.383e-10 & 4.00 &  3.300e-08 & 3.00\\
         128 & 1.700e-08 & 3.00 &  4.614e-11 & 4.00 &  4.124e-09 & 3.00\\
         256 & 2.125e-09 & 3.00 &  2.883e-12 & 4.00 &  5.155e-10 & 3.00\\
         512 & 2.657e-10 & 3.00 &  1.803e-13 & 4.00 &  6.444e-11 & 3.00
      \end{tabular}
     \end{center}
  \end{table}
  the convergence behaviour of $U_h$ for $N=M$ and polynomial 
  degrees $q=p+1=2$ and $q=p+1=3$.
  Note that we also show the norm $\tnorm{U-U_h}_\rho$ estimated by a refined 
  quadrature rule in the last columns.
  The estimated rates of convergence support our theoretical result in 
  Theorem~\ref{thm:cts}, that the error $E$ is of order $\min\{p,q+1\}$. 
  For odd polynomial degrees $p$ the component $\tnorm{U-U_h}_{Q,\rho}$ shows a convergence 
  order of one order higher, hinting at a superconvergence property.
  
  In Table~\ref{tab:prob1_2}     
  \begin{table}[ht]
     \caption{Convergence rates for $E(U-U_h)$ of problem \eqref{eq:prob1} 
              and several polynomial orders\label{tab:prob1_2}}
     \begin{center}   
      \begin{tabular}{r|ccccc}
          $p\setminus q$   & 1 & 2 & 3 & 4 & 5\\\hline
                         1 & 2 & 2 & 2 & 2 & 2\\
                         2 & 2 & 2 & 2 & 2 & 2\\
                         3 & 2 & 3 & 4 & 4 & 4\\
                         4 & 2 & 3 & 4 & 4 & 4\\
                         5 & 2 & 3 & 4 & 5 & 6
      \end{tabular}
     \end{center}
  \end{table}
  the estimated convergence rates for all combinations of polynomial degrees
  $\{p,\,q\}\subseteq\{1,\dots,5\}$ are given. Clearly the rates for even $p$ follow the 
  predicted $\min\{p,q+1\}$, while for odd $p$ the rates are $\min\{p+1,q+1\}$.
  Thus there might be dragons\footnote{superconvergence phenomena}.
  
  Let us modify problem \eqref{eq:prob1}, by taking $\Omega=\left[-\frac{3\pi}{2},\frac{3\pi}{2}\right]$,
  $\Omega_h= \left[-\frac{3\pi}{2},0\right]$, $\Omega_p=\left[0,\frac{3\pi}{2}\right]$
  and right-hand sides only in $L^2$.
  To be more precise, let
  \begin{align*}
   f(t,x) &= \id{\R_{\geq0}}(t)\bigg(
                   -(2\e^t-t-1)\id{\left( -\frac{\pi}{2},0\right)}(x)\cos(x)+
                       \e^t\left( \id{\left( \frac{\pi}{2},\frac{3\pi}{2}\right)}(x)-\id{\left(0,\frac{\pi}{2}\right)}(x)\right)\cos(x)+\\&\hspace*{2.5cm}
                       \id{(0,\pi)}(x)-\id{\left(\pi,\frac{3\pi}{2}\right)}(x)\bigg),\\
   g(t,x) &= \id{\R_{\geq0}}(t)\bigg( 
                     \id{(0,\pi)}(x)x+\id{\left( \pi,\frac{3\pi}{2}\right)}(x)(2\pi-x)\\&\hspace*{2.5cm}
                     -(\e^t-1)(\id{\left( \frac{\pi}{2},\frac{3\pi}{2}\right)}(x) -\id{\left( 0,\frac{\pi}{2}\right)}(x))\sin(x)
              \bigg).
  \end{align*}
  Figure~\ref{fig:prob2_fg}
  \begin{figure}[ht]
    \begin{center}
     \includegraphics[width=0.4\textwidth]{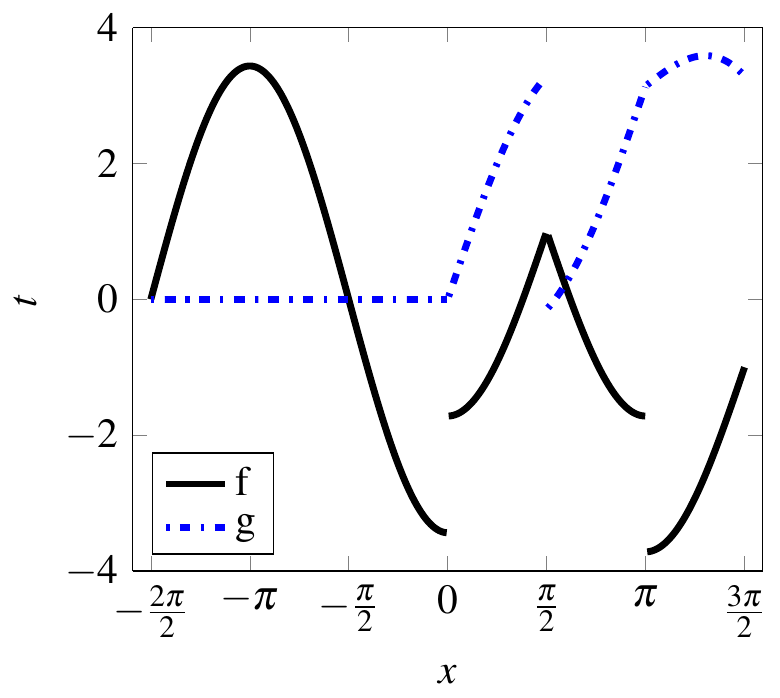}
    \end{center}
    \caption{Right-hand sides $f$ and $g$ of modified problem \eqref{eq:prob1}\label{fig:prob2_fg}}
  \end{figure}
  shows the right-hand sides $f$ and $g$ for $t=1$.
  
  Again the exact solution can be found and is given by
  \begin{align*}
      u(t,x) &= \id{\R_{\geq0}}(t)(\e^t-1)(\id{\left( \frac{\pi}{2},\frac{3\pi}{2}\right)}(x)-\id{\left(-\frac{3\pi}{2},\frac{\pi}{2}\right)}(x))\cos(x),\\
      v(t,x) &= \id{\R_{\geq0}}(t)\left( 
                     -(\e^t-t-1)\id{\left( -\frac{3\pi}{2},0\right) }(x)\sin(x)+
                     \id{(0,\pi)}(x)x+\id{\left( \pi,\frac{3\pi}{2}\right)}(x)(2\pi-x)
                \right) .
  \end{align*}
  Note that $u$ and $v$ are non-differentiable, but piece-wise smooth.
  Figure~\ref{fig:prob2_uv}
  \begin{figure}[ht]
    \begin{center}
     \includegraphics[width=0.4\textwidth]{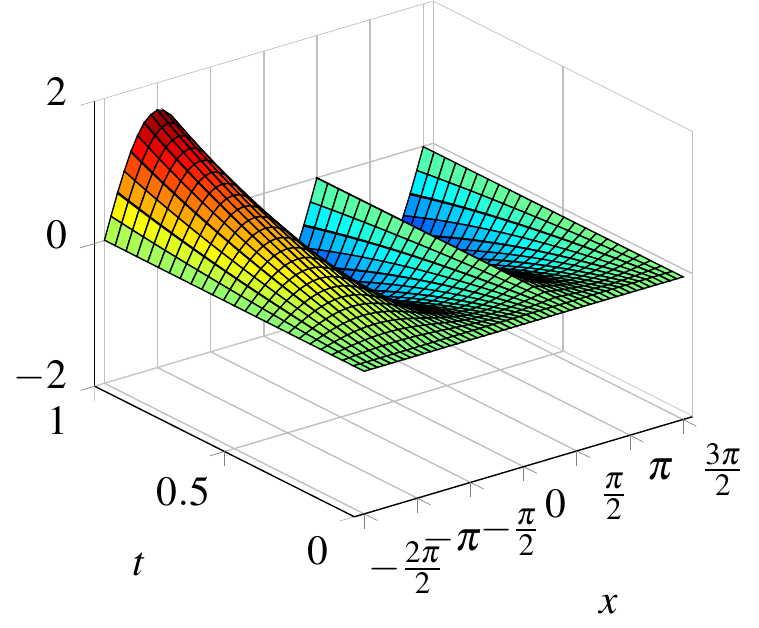}\quad
     \includegraphics[width=0.4\textwidth]{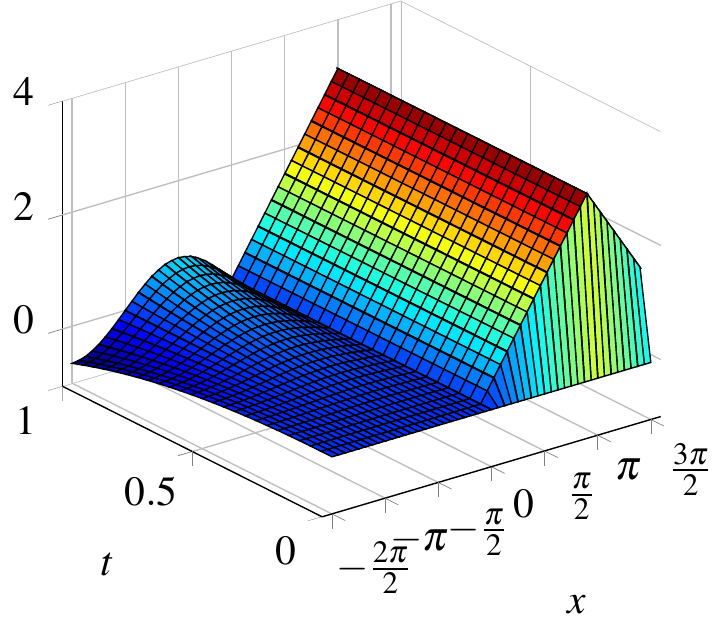}
    \end{center}
    \caption{Solution $u$ (left) and $v$ (right) of the modified problem \eqref{eq:prob1}\label{fig:prob2_uv}}
  \end{figure}
  shows the solutions for $t\in[0,1]$.
  
  Note that a priori, we impose no transmission condition. However, as in 
  \cite[Remark 3.2]{W_StH16}, they can be derived for $u$ satisfying 
  \eqref{eq:prob1} as
  \[
    u(t,0+) = u(t,0-),\quad
    \partial_x u(t,0+) = \int_0^t\partial_x u(s,0-)ds.
  \]
  For the numerical solution we use $T=1$, an equidistant mesh of $M$ 
  cells in time and an equidistant mesh of $N$ cells in space, thus $\tau=1/M$ and $h=1/N$. 
  In order to capture the jumps of $f$ and $g$, and to resolve the boundary 
  $S=\overline\Omega_h\cap\overline\Omega_p=\{0\}$ we use an equidistant mesh in 
  space with the number of cells $N$ divisible by 6.
  Note that we can use $\rho=1$ for the given solution $u$. 
  
  Defining
  \[
    E_{\sup}(v)^2 \coloneqq \sup_{t\in[0,T]}\scp{M_0v(t),v(t)},\quad
    E(v)^2        \coloneqq \sup_{t\in[0,T]}\scp{M_0v(t),v(t)}+\tnorm{v}_{Q,\rho}^2
  \]
  we consider in Table~\ref{tab:prob2}
  \begin{table}[ht]
     \caption{Convergence results for $U-U_h$ of problem modified \eqref{eq:prob1}
              \label{tab:prob2}}
     \begin{center}   
      \begin{tabular}{rllllll}
         $N=M$ & 
         \multicolumn{2}{c}{$E_{\sup}(U-U_h)$} &
         \multicolumn{2}{c}{$\tnorm{U-U_h}_{Q,\rho}$} & 
         \multicolumn{2}{c}{$\tnorm{U-U_h}_{\rho}$}\\
         \hline
         \multicolumn{7}{c}{$p=2$, $q=1$}\\
         \hline
          12 & 2.159e-02 &      &  3.953e-03 &      &  4.110e-03 &     \\
          24 & 5.490e-03 & 1.98 &  1.017e-03 & 1.96 &  1.055e-03 & 1.96\\
          48 & 1.409e-03 & 1.96 &  2.557e-04 & 1.99 &  2.651e-04 & 1.99\\
          96 & 3.577e-04 & 1.98 &  6.400e-05 & 2.00 &  6.637e-05 & 2.00\\
         192 & 9.010e-05 & 1.99 &  1.601e-05 & 2.00 &  1.660e-05 & 2.00\\
         384 & 2.261e-05 & 1.99 &  4.002e-06 & 2.00 &  4.150e-06 & 2.00\\
         768 & 5.662e-06 & 2.00 &  1.001e-06 & 2.00 &  1.037e-06 & 2.00\\
         \hline
         \multicolumn{7}{c}{$p=3$, $q=2$}\\
         \hline
          12 & 1.334e-04 &      &  2.629e-05 &      &  2.734e-05 &     \\
          24 & 5.921e-06 & 4.49 &  7.802e-07 & 5.07 &  1.220e-06 & 4.49\\
          48 & 5.585e-07 & 3.41 &  2.408e-08 & 5.02 &  1.197e-07 & 3.35\\
          96 & 6.981e-08 & 3.00 &  7.500e-10 & 5.00 &  1.468e-08 & 3.03\\
         192 & 8.726e-09 & 3.00 &  2.343e-11 & 5.00 &  1.833e-09 & 3.00\\
         384 & 1.091e-09 & 3.00 &  7.329e-13 & 5.00 &  2.291e-10 & 3.00\\
         768 & 1.363e-10 & 3.00 &  2.474e-14 & 4.89 &  2.864e-11 & 3.00
      \end{tabular}
     \end{center}
  \end{table}
  \begin{table}[ht]
     \caption{Convergence rates for $E(U-U_h)$ of the modified problem \eqref{eq:prob1} 
              and several polynomial orders\label{tab:prob2_2}}
     \begin{center}   
      \begin{tabular}{r|ccccc}
          $p\setminus q$   & 1 & 2 & 3 & 4 & 5\\\hline
                         1 & 2 & 3 & 3 & 3 & 3\\
                         2 & 2 & 2 & 2 & 2 & 2\\
                         3 & 2 & 3 & 5 & 5 & 5\\
                         4 & 2 & 3 & 4 & 4 & 4\\
                         5 & 2 & 3 & 4 & 7 & 7
      \end{tabular}
     \end{center}
  \end{table}
   we observe a convergence behaviour similar to the previous smooth case.

  \subsection{Changing type system -- two space dimensions}
  This time we consider a problem with unknown solution.
  Let $\Omega=(0,1)^2\subset\R^2$, $\Omega_h=\left( \frac{1}{4},\frac{3}{4}\right)^2$, 
  $\Omega_e=\Omega\setminus\bar\Omega_h$ and $\Omega_p=\emptyset$. 
  The problem is given on $(0,T)\times\Omega$ by
  \begin{gather}\label{eq:prob3_pde}
     \left( 
      \partial_t
      \begin{pmatrix}
       1 & 0\\
       0 & \id{h}
      \end{pmatrix}
      +\begin{pmatrix}
        0 & 0\\
        0 & \id{p}
       \end{pmatrix}
      +\begin{pmatrix}
           0 & \Div \\
           \grad_0 & 0
       \end{pmatrix}
     \right)\begin{pmatrix}
             u\\
             v
            \end{pmatrix}
       =\begin{pmatrix}
               f\\
               0
        \end{pmatrix},
  \end{gather}
  where
  \[
      f(t,x)=2\sin(\pi t)\id{\R_{<1/2}\times\R}(x).
  \]
  For $T=1.875$ Figure~\ref{fig:prob3} 
%
  \begin{figure}[tb]
     \begin{center}
      \includegraphics[width=0.3\textwidth]{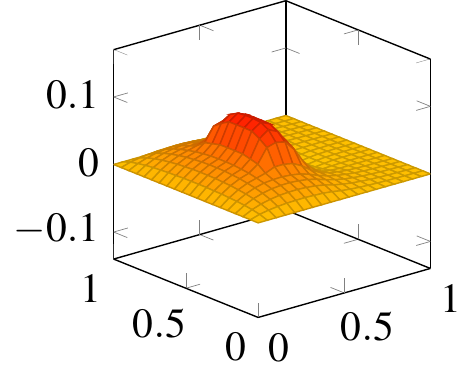}\quad
      \includegraphics[width=0.3\textwidth]{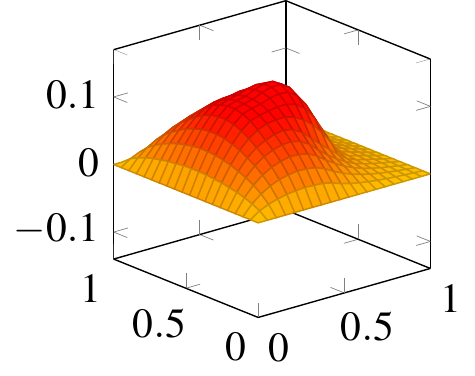}\quad
      \includegraphics[width=0.3\textwidth]{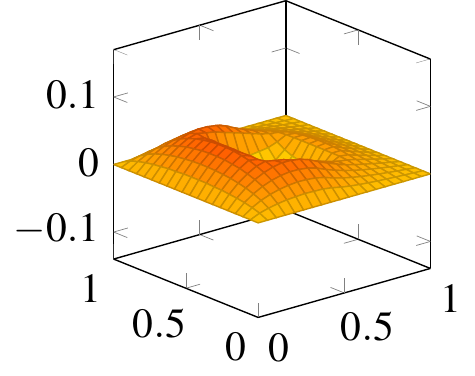}\\
      \includegraphics[width=0.3\textwidth]{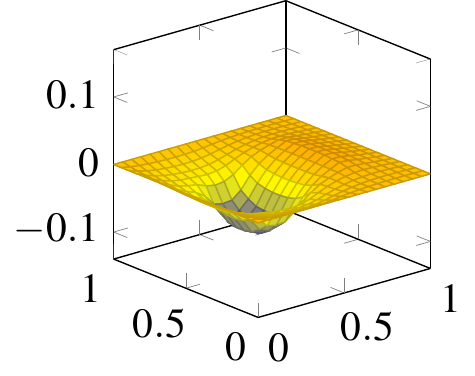}\quad
      \includegraphics[width=0.3\textwidth]{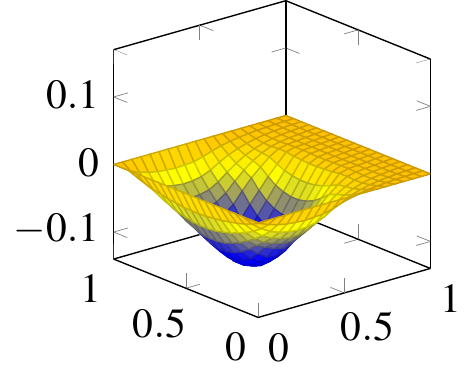}\quad
      \includegraphics[width=0.3\textwidth]{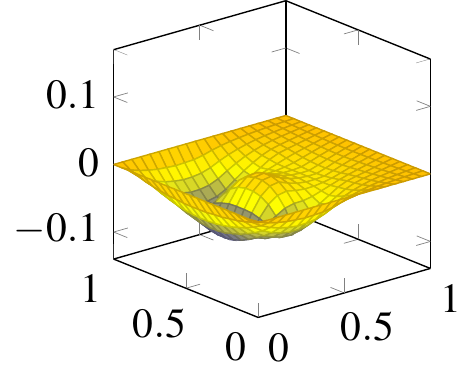}
     \end{center}
     \caption{Solution $u$ at times $t=5k/16$ for $k\in\{1,\dots, 6\}$ (top left to bottom right) of problem \eqref{eq:prob3_pde} for $T=1.875$}\label{fig:prob3}
  \end{figure}
  
   \begin{table}[tb]
     \caption{Convergence results for $\tilde U-U_h$ of problem \eqref{eq:prob3_pde}
              \label{tab:prob3}}
     \begin{center}   
      \begin{tabular}{rllllll}
         $N=M$ & 
         \multicolumn{2}{c}{$E_{\sup}(\tilde U-U_h)$} &
         \multicolumn{2}{c}{$\tnorm{\tilde U-U_h}_{Q,\rho}$} & 
         \multicolumn{2}{c}{$\tnorm{\tilde U-U_h}_{\rho}$}\\
         \hline
         \multicolumn{7}{c}{$p=2$, $q=1$}\\
         \hline
          4 & 1.660e-02 &     &  8.121e-03 &     &  8.703e-03&     \\
          8 & 5.595e-03 &1.57 &  2.425e-03 &1.74 &  2.781e-03& 1.65\\
         16 & 1.666e-03 &1.75 &  7.445e-04 &1.70 &  8.517e-04& 1.71\\
         32 & 5.260e-04 &1.66 &  2.790e-04 &1.42 &  3.012e-04& 1.50\\
         64 & 1.926e-04 &1.45 &  1.300e-04 &1.10 &  1.331e-04& 1.18\\
         \hline
         \multicolumn{7}{c}{$p=3$, $q=2$}\\
         \hline
          4 & 4.895e-03 &      & 1.778e-03 &      & 2.028e-03 &    \\
          8 & 1.117e-03 &2.13  & 5.510e-04 &1.69  & 5.748e-04 &1.82\\
         16 & 4.015e-04 &1.48  & 2.414e-04 &1.19  & 2.419e-04 &1.25\\
         32 & 1.430e-04 &1.49  & 1.175e-04 &1.04  & 1.175e-04 &1.04\\
         64 & 5.245e-05 &1.45  & 5.075e-05 &1.21  & 5.072e-05 &1.21
      \end{tabular}
     \end{center}
  \end{table}
  shows some snapshots of the component $u$ of the solution $U$,
  approximated by a numerical simulation.
  
  In order to investigate the error-behaviour upon refinement of the discretisation, we use 
  a numerically computed reference solution $\tilde U$ instead of the real one $U$. For this we set $T=1$ and use 
  an equidistant mesh of 128$\times$128 rectangular cells in space and 128 cells in time, and 
  polynomial degrees $p=3$ and $q=2$. Thus $u$ is approximated in space by piece-wise $\QS_3$ elements,
  $v$ by $RT_2$-elements and both in time by $\PS_2$-elements.
  In Table~\ref{tab:prob3} we see the results of our numerical simulation 
  for two pairs of polynomial order. We observe, that the error rates are independent of the polynomial order
  and furthermore less than the optimal orders given in Theorem~\ref{thm:cts}. The reason for this
  decrease in convergence order lies in the reduced regularity of the solution to this given problem.
  The interior boundaries where the type of the problem changes introduces corners, where it is very likely for
  singular solution components to arise.

  \section{Appendix -- On the Gau\ss--Radau Quadrature}\label{s:app}
  
  In this appendix we shall gather some results on the right-sided Gau\ss--Radau quadrature, which are known in principle, but are included for the convenience of the reader. We adopted the rational given in \cite{L04}. For this, we introduce a set of weighting functions:
  \[
    W \coloneqq \{ w\in L^1(-1,1): w>0 \text{ a.e.}\}.
  \]
  Note that the bilinear form
  \[
     \langle \cdot,\cdot\rangle_w \colon (f,g)\mapsto \int_{-1}^1 f(x)g(x)w(x) dx
  \]
 introduces a scalar product on its natural domain
  \[
    D\coloneqq \{ f\in L^1_{\mathrm{loc}}(-1,1); \int_{(-1,1)} |f(x)|^2 w(x) \dd x < \infty \}.   
  \] 
  Furthermore, for all $w\in W$ we set $\tilde w \colon x\mapsto (1-x)w(x)$. We observe $\tilde w \in W$. Throughout, let $q\in \N$.
  \begin{deff} Let $w\in W$. A pair $(\omega,r)=((\omega_j)_{j\in\{0,\ldots,q\}},(r_j)_{j\in\{0,\ldots,q\}})\in \R^{q+1}\times [-1,1]^{q+1}$ is called \emph{(right-sided) $w$-Gau\ss--Radau quadrature (of order $q$)}, if $-1\leq r_0\leq r_1 \leq \cdots \leq r_q=1$ and for all $p\in \mathcal{P}_{2q}(-1,1)$ we have
  \[
    \int_{-1}^1p(x)w(x)dx = \sum_{j=0}^q \omega_j p(r_j).
  \]
  \end{deff}
  \begin{prop}\label{p:uni} Let $w\in W$, $(\omega,r)$ a $w$-Gau\ss--Radau quadrature. Then the following properties are satisfied:
  \begin{enumerate}
   \item\label{u1} the set $\{r_j; j\in\{0,\ldots,q\} \}$ consists of $q+1$ elements;
   \item\label{u2} for all $j\in \{0,\ldots,q\}$ we have $0<\omega_j\leq \int_{(-1,1)} w(x)dx$;
   \item\label{u3} for all $j\in \{0,\ldots,q\}$ we have with $I_j(x)\coloneqq \prod_{k\in\{0,\ldots,q\}\setminus\{j\}}\frac{x-r_k}{r_j-r_k}$ \[ \omega_j =\int_{-1}^1I_j(x)w(x)dx; \]
   \item\label{u4} if $(\omega^{(1)},r^{(1)})$ is a $w$-Gau\ss-quadrature, then $(\omega,r)=(\omega^{(1)},r^{(1)})$.
  \end{enumerate}
  \end{prop}
  \begin{proof}
    For the proof \eqref{u1}, we assume that $Z\coloneqq\{r_j; j\in\{0,\ldots,q\} \}$ has strictly less than $q+1$ elements. Then we find a polynomial $p$ of degree at most $q$ such that $Z$ is the set of zeros of $p$. Furthermore, for $z\in Z$ there exists a polynomial $p_z$ of degree at most $q-1$ with the property $p_z(z)=1$ and $p_z=0$ on $Z\setminus \{z\}$. Thus, by exactness of the quadrature and $w>0$ a.e., we obtain
    \[
       0< \int_{-1}^1 p_z(x)^2 w(x) dx = \sum_{j\in \{k; r_k=z\}} \omega_j p_z(r_j)^2=\sum_{j\in \{k; r_k=z\}} \omega_j.
    \]
    Consequently, as $p^2$ has degree at most $2q$, we infer
    \begin{multline*}
       0< \int_{-1}^1 p(x)^2 w(x) dx = \sum_{j\in \{0,\ldots,q\}} \omega_j p(r_j)^2\\= \sum_{z\in Z}\sum_{j\in \{k; r_k=z\}}\omega_j p(r_j)^2=\sum_{z\in Z} p(z)^2\sum_{j\in \{k; r_k=z\}}\omega_j=0,  
    \end{multline*}
    a contradiction.
    
    Next, for \eqref{u2}, by \eqref{u1}, we observe that $I_j(x)$ in \eqref{u3} is well-defined for all $j\in \{0,\ldots,q\}$. Thus, for $j\in \{0,\ldots,q\}$, we obtain
    \[
       0<\int_{-1}^1 I_j(x)^2w(x)dx = \omega_j.
    \]
    Hence, we get for all $j\in\{0,\ldots,q\}$
    \[
      \omega_j \leq \sum_{\ell \in \{0,\ldots,q\}} \omega_\ell = \int_{-1}^1w(x)dx.
    \]
    The proof of \eqref{u3} is obvious.
    
    For the proof of \eqref{u4}, by the Gram--Schmidt orthonormalization procedure, we choose a polynomial $p_q\in \mathcal{P}_q(-1,1)$ such that $p_q$ is orthogonal to $\mathcal{P}_{q-1}(-1,1)$ with respect to $\langle \cdot,\cdot\rangle_{\tilde w}$. 
    
    Let $p\in \mathcal{P}_{q-1}(-1,1)$. Then the polynomial $(1-\cdot)p p_q$ has degree at most $2q$. Thus, by the choice of $p_q$ and the exactness of the quadrature, we obtain
    \begin{align*}
       0 &= \langle p,p_q\rangle_{\tilde w}
       \\ & = \int_{-1}^1 (1-x)p(x)p_q(x)w(x)dx
       \\ &= \sum_{j\in\{0,\ldots,q\}}\omega_j(1-r_j)p(r_j)p_q(r_j)
       \\ &=  \sum_{j\in\{0,\ldots,q-1\}}\omega_j(1-r_j)p(r_j)p_q(r_j)=\omega_i(1-r_i)p_q(r_i),
    \end{align*}
   if $p=I_i$ for one $i\in\{0,\ldots,q-1\}$. From \eqref{u1} and \eqref{u2}, we get $\omega_i(1-r_i)\neq 0$ (recall that $r_q=1$). Hence, $p_q(r_i)=0$ for all $i\in \{0,\ldots,q-1\}$. As $p_q$ has degree $q$, we obtain $r=r^{(1)}$. Hence, the assertion follows from the formula for $\omega$ in statement \eqref{u3}.
  \end{proof}
  
  The next proposition is concerned with the existence of the quadrature:
  \begin{prop}\label{p:exi} Let $w\in W$, $p_q\in \mathcal{P}_q(-1,1)$ such that $p_q\bot \mathcal{P}_{q-1}(-1,1)$ with respect to $\langle\cdot,\cdot\rangle_{\tilde w}$. Then the following assertions hold true:
  \begin{enumerate}
   \item\label{e1} $p_q$ has exactly $q$ distinct real roots all contained in $(-1,1)$;
   \item\label{e2} if $-1<r_0<r_1<\ldots<r_{q-1}<1$ denote the roots of $p_q$, then a $w$-Gau\ss--Radau quadrature is given by $((\omega_j)_{j\in\{0,\ldots,q\}},(r_j)_{j\in\{0,\ldots,q\}})$, where $r_q=1$ and \[ \omega_j =\int_{-1}^1 \prod_{k\in\{0,\ldots,q\}\setminus\{j\}}\frac{x-r_k}{r_j-r_k}w(x)dx \quad(j\in\{0,\ldots,q\}). \]
   \end{enumerate}
  \end{prop}
  \begin{proof} For the proof of \eqref{e1}, let $O\subseteq (-1,1)$ be the set of roots of $p_q$ contained in $(-1,1)$ with odd multiplicity. Define $p(x)\coloneqq \prod_{z\in O} (x-z)$, $x\in [-1,1]$ ($p=1$, if $O=\emptyset$). We are done, once we show that $|O|=q$. Assume $|O|<q$. Then $p\in \mathcal{P}_{q-1}(-1,1)$. Moreover, the polynomial $p^*\colon x\mapsto (1-x)p(x)p_q(x)$ is non-zero and has no sign-change in $(-1,1)$. Without restriction, we assume $p^*\geq 0$. From
  \[
     0<\int_{-1}^1 p^*(x)w(x)dx=\langle p,p_q\rangle_{\tilde w}=0,
  \]
  we obtain a contradiction.
  
  In order to proof \eqref{e2}, let $p\in \mathcal{P}_{2q}(-1,1)$. We find polynomials $f\in \mathcal{P}_{q-1}(-1,1)$ and $g\in \mathcal{P}_q(-1,1)$ with the property $p= (x\mapsto f(x)(1-x)p_q(x))+ g$. Since $g$ is of degree at most $q$, we obtain
  \[
     g(x) = \sum_{j=0}^q g(r_j)\prod_{k\in\{0,\ldots,q\}\setminus\{j\}}\frac{x-r_k}{r_j-r_k}\quad(x\in(-1,1)). 
  \]
Then, using that $\langle f,p_q\rangle_{\tilde w}=0$, we compute
  \begin{align*}
     \int_{-1}^1 p w & = \int_{-1}^1 (f(x)(1-x)p_q(x)+ g(x)) w(x)dx
     \\ & = \int_{-1}^1 g(x) w(x)dx
     \\ & = \int_{-1}^1 \sum_{j=0}^q g(r_j)\prod_{k\in\{0,\ldots,q\}\setminus\{j\}}\frac{x-r_k}{r_j-r_k} w(x)dx
     \\ & = \sum_{j=0}^q g(r_j) \int_{-1}^1 \prod_{k\in\{0,\ldots,q\}\setminus\{j\}}\frac{x-r_k}{r_j-r_k} w(x)dx
     \\ & = \sum_{j=0}^q g(r_j) \omega_j.
  \end{align*}
Since $p(r_j)=f(r_j)(1-r_j)p_q(r_j)+ g(r_j)=g(r_j)$ for all $j\in \{0,\ldots,q\}$, the assertion is proved.
  \end{proof}
  We address the continuous dependence of the Gau\ss--Radau points on the weighting function as follows.
  \begin{thm}\label{t:cd} The mapping
    \begin{equation}\label{e:cd}
      (W,\|\cdot\|_{L^1(-1,1)}) \to \R^{q+1}\times (-1,1]^{q+1}, w\mapsto (\omega(w),r(w))
    \end{equation}
    is continuous, where $(\omega(w),r(w))$ denotes the $w$-Gau\ss--Radau quadrature.   
  \end{thm}
  \begin{proof}
    Let $(w_n)_{n\in \N}$ be a sequence in $W$ and $w\in W$ such that 
    $w_n\to w$ in $L^1(-1,1)$. By definition and by Theorem \ref{p:uni}, 
    \[
    (\omega(w_n),r(w_n))\in [0,\sup_{k\in \N} 
    \|w_k\|_{L^1(-1,1)}]^{q+1}\times [-1,1]^{q+1}   
    \]
    for all $n\in \N$. Thus, there exists a convergent subsequence for which we 
    re-use the name with limit $(\overline{\omega},\overline{r})$. Note that 
    $\overline{r}_q=1$. Next, let $p\in \mathcal{P}_{2q}(-1,1)$. Then, for $n\in 
    \N$, we obtain
    \begin{align*}
      \int_{-1}^1 p w &= \lim_{n\to \infty} \int_{-1}^1 p w_n\\ 
      & =\lim_{n\to \infty} \sum_{j=0}^q \omega(w_n)_j p(r(w_n)_j)\\ 
      &= \sum_{j=0}^q \overline{\omega}_j p(\overline{r}_j).
    \end{align*}
    Hence, by Theorem \ref{p:uni}, we infer 
    $(\overline{\omega},\overline{r})=(\omega(w),r(w))$, which eventually implies 
    the assertion.
  \end{proof}
  \begin{cor}\label{c:chi1} For $\tau\in \R$ denote $w_\tau\colon x\mapsto \exp(-\rho \tau(x+1))(\in W)$ and let $(\omega^{(\tau)},r^{(\tau)})$ be the $w_\tau$-Gau\ss--Radau quadrature. For $\tau\in\R$, let $\chi_\tau\in \mathcal{P}_{q+1}(-1,1)$ such that\[
    \chi_\tau(r^{(\tau)}_j)=0\quad(j\in\{0,\ldots,q\}),\, \chi_\tau(-1)=1.
  \]
  Then, for every compact set $K\subset \R$, we have
  \[
     \sup_{\tau\in K} \int_{-1}^1 \chi_\tau^2w_\tau <\infty.
  \]  
  \end{cor}
  \begin{proof}
   Assume by contradiction that there exists $(\tau_n)_n$ convergent to some $\tau$ with the property
   \[
     \int_{-1}^1 \chi_{\tau_n}^2w_{\tau_n}\to \infty.
   \]
   Using Theorem \ref{t:cd}, we compute for $n\in \N$
   \begin{align*}
     \int_{-1}^1 \chi_{\tau_n}^2w_{\tau_n} & = \int_{-1}^1 \prod_{j=0}^q \Big(\frac{x-r^{(\tau_n)}_j}{-1-r^{(\tau_n)}_j}\Big)^2 w_{\tau_n}(x)dx
     \\ & \leq \int_{-1}^1 \prod_{j=0}^{q-1} \Big(\frac{x-r^{(\tau_n)}_j}{-1-r^{(\tau_n)}_j}\Big)^2 w_{\tau_n}(x)dx
     \\ & = \sum_{\ell=0}^q \omega^{(\tau_n)}_\ell \prod_{j=0}^{q-1} \Big(\frac{r^{(\tau_n)}_\ell-r^{(\tau_n)}_j}{1+r^{(\tau_n)}_j}\Big)^2
     \\ & =  \omega^{(\tau_n)}_q \prod_{j=0}^{q-1} \Big(\frac{1-r^{(\tau_n)}_j}{1+r^{(\tau_n)}_j}\Big)^2
     \\ & \to \omega^{(\tau)}_q \prod_{j=0}^{q-1} \Big(\frac{1-r^{(\tau)}_j}{1+r^{(\tau)}_j}\Big)^2\quad(n\to\infty).
   \end{align*}
   But,
   \[
     0\leq \omega^{(\tau)}_q \prod_{j=0}^{q-1} \Big(\frac{1-r^{(\tau)}_j}{1+r^{(\tau)}_j}\Big)^2=\int_{-1}^1 \prod_{j=0}^{q-1} \Big(\frac{x-r^{(\tau)}_j}{-1-r^{(\tau)}_j}\Big)^2 w_{\tau}(x)dx<\infty,
   \]
  which contradicts the assumption.
  \end{proof}
  The next two corollaries are the ones needed in Subsection \ref{su:err1}, that is, for the error estimate with respect to the time-discretization. Beforehand, we introduce for a bounded interval $I\subseteq \R$ the mapping
  \begin{align*}
      \phi_I \colon (-1,1) &\to I,
      \\                   x&\mapsto \frac{a+b}{2}+\frac{b-a}{2}x,
  \end{align*}
  where $a\coloneqq \inf I$, $b\coloneqq \sup I$. Further, we set $|I|\coloneqq b-a$. 
  \begin{cor}\label{c:chi2} For $\tau\in \R$ let $\chi_\tau$ be as in Corollary \ref{c:chi1}. Let $K\geq 0$. Then
  \[
     \sup_{I\subseteq \R\text{ interval}, |I|\leq K} \frac{1}{|I|}\int_I \big(\chi_{|I|}(\phi_I^{-1}(t))\big)^2e^{-2\rho(t-\inf I)}dt<\infty.
  \]   
  \end{cor}
  \begin{proof}
    For $I=(a,b)\subseteq \R$ we compute
    \begin{align*}
      \frac{1}{b-a}\int_{a}^b \big(\chi_{|I|}(\phi_I^{-1}(t))\big)^2 e^{-2\rho(t-a)} dt
      &=\frac{1}{b-a}\int_{-1}^1 \big(\chi_{|I|}(x)\big)^2 e^{-2\rho(\phi_I(t)-a)}\phi_I'(x) dx
     \\ &=\frac{1}{2}\int_{-1}^1 \big(\chi_{|I|}(x)\big)^2 e^{-\rho((b-a)(x+1))} dx.
    \end{align*}
   Hence, the assertion follows from Corollary \ref{c:chi1}.
  \end{proof}
  The next corollary is concerned with the lowest Gau\ss--Radau point for different weights:
  \begin{cor}\label{c:t0} For $\tau\in \R$ let $w_\tau$ and $(\omega^{(\tau)},r^{(\tau)})$ be given as in Corollary \ref{c:chi1}. Let $T>0$. Then there exists $c>0$ such that for all intervals $I\subseteq \R$ with $|I|\leq T$ and $0\leq\tau \leq T$ we have
  \[
     \phi_I(r^{(\tau)}_0)-\inf I \geq c |I|.
  \]   
  \end{cor}
  \begin{proof}
    We observe that $\R\ni \tau \mapsto \big(t\mapsto e^{-\rho \tau(t+1)}\big)\in (W,\|\cdot\|_{L^1(-1,1)})$ is continuous. Hence, the set 
    \[
       \{ \big(t\mapsto e^{-\rho \tau(t+1)}\big); \tau\in [0,T]\}\subseteq (W,\|\cdot\|_{L^1(-1,1)})
    \]
    is compact. Thus, by the continuous dependence of the Gau\ss--Radau point on the weighting function (see \eqref{e:cd}), we obtain that 
    \[
       \{ (\omega^{(\tau)},r^{(\tau)}); \tau\in [0,T]\}\subseteq \R^{q+1}\times (-1,1]^{q+1}
    \]
    is compact, as well. In particular, there exists $c>0$ with the property $r^{(\tau)}_0-(-1)\geq c$. Hence, we obtain for all $\tau\in[0,T]$ and intervals $I\subseteq \R$ with $|I|\leq T$
    \[
       \phi_I(r^{(\tau)}_0)-\inf I = \phi_I(r^{(\tau)}_0)-\phi_I(-1) = |I|(r^{(\tau)}_0-(-1)) \geq c |I|.\qedhere
    \]
  \end{proof}

\end{document}